\documentclass{svjour2}                    
\smartqed  
\usepackage{graphicx}
\usepackage{amssymb,amsmath}
\newcommand{\fracs}[2]{\leavevmode\kern.1em\raise.5ex\hbox{\the\scriptfont0 #1}\kern-.2em/\kern-.15em\lower.25ex\hbox{\the\scriptfont0 #2}} 
\newcommand{\fracss}[2]{\leavevmode\kern.1em\raise.5ex\hbox{\the\scriptfont0 #1}\kern-.3em/\kern-.15em\lower.25ex\hbox{\the\scriptfont0 #2}} 
\renewcommand{\vec}[1]{{\bf #1}}
\newcommand{\nablauh}{\nabla_{\!h}}
\newcommand{\pizw}{\Pi_N(z,w)}

\newcommand{\bracketmatrix}[1]{\left[\begin{matrix}#1\end{matrix}\right]}

\newcommand{\dform}[2]{d#1_#2}
\newcommand{\del}{\partial}
\newcommand{\delb}{\bar\partial}
\newcommand{\dz}{dz}

\newcommand{\zb}{{\bar z}}

\newcommand{\w}{\wedge}

\newcommand{\wb}{\bar w}

\newcommand{\matpad}{\hbox to 1em{}}
\newcommand{\e}{\varepsilon}
\newcommand{\intr}{\int_{\mathbb{R}}}

\newcommand{\intc}{\int_{\mathbb{C}}}
\newcommand{\intct}{\int_{\mathbb{C}^2}}
\newcommand{\intcf}{\int_{\mathbb{C}^4}}

\newcommand{\intrf}{\int_{\mathbb{R}^4}}

\newcommand{\lime}{\lim_{\e_j\to0}}
\newcommand{\limep}{\lim_{\e'_{\!j}\to0}}

\newcommand{\h}[1]{h_{#1}}

\newcommand{\zz}[1]{\zeta_{#1}}
\newcommand{\ee}[1]{\e_{#1}}

\newcommand{\limes}{\lim_{\vec\e\to\vec0}}
\newcommand{\je}{J_{\vec\e}}

\newcommand{\thelor}{-\langle\Lambda^{-1}(0,r)v,v\rangle}
\newcommand{\thelt}{-\langle\Lambda^{-1}(t)v,v\rangle}

\newcommand{\etothelt}{e^{\thelt}}
\newcommand{\expy}[1]{\mathop{\rm exp}\nolimits\left[#1\right]}

\newcommand{\realpart}{\mathop{\rm Re}\nolimits}

\newcommand{\bigo}[1]{{O}\!\left(#1\right)}

\newcommand{\detildi}{{\det(i\Lambda D-I)}}

\renewcommand{\Im}{\mathop{\rm Im}\nolimits}

\renewcommand{\l}{\ell}

\newcommand{\half}{{\kern.1em\raise.5ex\hbox{\the\scriptfont0 1}\kern-.2em/\kern-.15em\lower.25ex\hbox{\the\scriptfont0 2}}}
\newcommand{\threehalves}{{\kern.1em\raise.5ex\hbox{\the\scriptfont0 3}\kern-.2em/\kern-.15em\lower.25ex\hbox{\the\scriptfont0 2}}}

\newcommand{\rbreak}{\par\medskip\par}

\newcommand{\honer}{h_{1,\mathbb{R}}}
\newcommand{\honei}{h_{1,\mathbb{I}}}
\newcommand{\xoner}{x_{1,\mathbb{R}}}
\newcommand{\xonei}{x_{1,\mathbb{I}}}
\newcommand{\htwor}{h_{2,\mathbb{R}}}
\newcommand{\htwoi}{h_{2,\mathbb{I}}}
\newcommand{\xtwor}{x_{2,\mathbb{R}}}
\newcommand{\xtwoi}{x_{2,\mathbb{I}}}

\newcommand{\dst}{ \bracketmatrix{d\honer\cr d\honei\cr d\xoner\cr d\xonei\cr d\htwor\cr d\htwoi\cr d\xtwor\cr d\xtwoi}{\kern-.4em\raise 11 ex\hbox{$\scriptstyle\top$}} }

\newcommand{\E}[1]{{{\rm E}\left[#1\right]}}
\renewcommand{\P}[1]{{\rm P}\!\left[#1\right]}

\renewcommand{\lll}[1]{\lambda_{#1}}

\newcommand{\f}[1]{f_{#1}}
\newcommand{\F}[1]{F_{#1}}
\newcommand{\st}{\hbox{ such that }}

\newcommand{\z}[1]{z_{#1}}
\newcommand{\lsz}[1]{|z_{#1}|^2}
\newcommand{\lsw}[1]{|w_{#1}|^2}

\newcommand{\tr}{{\mathop{\rm tr}\nolimits\,}}
\newcommand{\dzb}{d\bar z}

\newcommand{\tsw}[1]{|s^{#1j}z_j|^2}

\newcommand{\biggg}[1]{{\hbox{$\left#1\vbox to 23.5pt{}\right.\n@space$}}}
\newcommand{\bigggl}{\mathopen\biggg}
\newcommand{\bigggr}{\mathclose\biggg}
\newcommand{\eeehack}[2]{
e
  \kern-.5em
  \raise 3.5 ex\hbox{$\biggl($}
  \raise 2.5 ex
  \hbox{
    \kern-.5em
    \vbox{
      \baselineskip=0pt plus 0pt
      \lineskip=1pt
      \hbox{$\scriptstyle-\lll#1\tsw#1$}
      \hbox{$\scriptstyle-\lll#2\tsw#2$}
    }
    \kern-.25em
  }
  \raise 3.5 ex \hbox{$\biggr)$}
}
\newcommand{\xxx}[2]{x_{#1#2}}
\newcommand{\usw}[1]{\big|\sum_ju_{j#1}z_j\big|^2}
\newcommand{\usww}[1]{\big|\sum_ju_{j#1}{w_j\over\sqrt{\lll j}}\big|^2}
\newcommand{\uswww}[1]{\big|\sum_j{t^\frac52u_{j#1}\over\sqrt{\lll j}}\,w_j\big|^2}

\newcommand{\eeeehack}[1]{
\expandafter\ifcase\number#1
\or
e
  \kern-.5em
  \raise 5.5 ex\hbox{$\bigggl($}
  \raise 2.5 ex
  \hbox{
    \kern-.5em
    \vbox{
      \baselineskip=0pt plus 0pt
      \lineskip=1pt
      \hbox{$\scriptstyle-\lll1\lsz1$}
      \hbox{$\scriptstyle-\lll2\lsz2$}
      \hbox{$\scriptstyle-\lll3\lsz3$}
      \hbox{$\scriptstyle-\lll4\lsz4$}
    }
    \kern-.25em
  }
  \raise 5.5 ex \hbox{$\bigggr)$}
\or
e
  \kern-.5em
  \raise 4.25 ex\hbox{$\biggl($}
  \raise 2.5 ex
  \hbox{
    \kern-.5em
    \vbox{
      \baselineskip=0pt plus 0pt
      \lineskip=1pt
      \hbox{$\scriptstyle-\xxx12\zzb1\z2-\xxx13\zzb1\z3-\xxx14\zzb1\z4$}
      \hbox{$\scriptstyle\phantom{-\xxx12\zzb1\z2}-\xxx23\zzb2\z3-\xxx24\zzb2\z4$}
      \hbox{$\scriptstyle\phantom{-\xxx12\zzb1\z2-\xxx12\zzb1\z2}-\xxx34\zzb3\z4$}
    }
    \kern-.25em
  }
  \raise 4.25 ex \hbox{$\biggr)$}
\or
e
  \kern-.5em
  \raise 5.5 ex\hbox{$\bigggl($}
  \raise 2.5 ex
  \hbox{
    \kern-.5em
    \vbox{
      \baselineskip=0pt plus 0pt
      \lineskip=1pt
      \hbox{$\scriptstyle-|\sqrt{\lll1}z_1|^2$}
      \hbox{$\scriptstyle-|\sqrt{\lll2}z_2|^2$}
      \hbox{$\scriptstyle-|\sqrt{\lll3}z_3|^2$}
      \hbox{$\scriptstyle-|\sqrt{\lll4}z_4|^2$}
    }
    \kern-.25em
  }
  \raise 5.5 ex \hbox{$\bigggr)$}
\or
e^{-\angleenclose{w,w}}
\or
e
  \kern-.5em
  \raise 5.5 ex\hbox{$\bigggl($}
  \raise 2.5 ex
  \hbox{
    \kern-.5em
    \vbox{
      \baselineskip=0pt plus 0pt
      \lineskip=1pt
      \hbox{$\scriptstyle-x_1^2-y_1^2$}
      \hbox{$\scriptstyle-x_2^2-y_2^2$}
      \hbox{$\scriptstyle-x_3^2-y_3^2$}
      \hbox{$\scriptstyle-x_4^2-y_4^2$}
    }
    \kern-.25em
  }
  \raise 5.5 ex \hbox{$\bigggr)$}
\or 
  e^{-\sum_{j=1}^4\lambda_jz_j\bar z_j}
\or 
  e^{-\sum_{j=1}^4\lambda_j|z_j|^2}
\or 
  e^{-\sum_{j=1}^4|\sqrt{\lambda_j}z_j|^2}
\or 
  e^{\!-\!\!\sum_{j=1}^4\!\!|w_j|^2}
\fi
}
\newcommand{\zzb}[1]{\bar z_{#1}}

\newcommand{\uu}[2]{u_{#1#2}}
\newcommand{\A}[4]{A^{#2}_{#4}}
\newcommand{\B}[5]{B^{#2}_{#4 #5}}
\newcommand{\C}[6]{C^{#2\,#3}_{#5 #6}}

\newcommand{\spp}{\Pi_1^H(z,w)|_{(\zeta_p,\zeta_{p'})}}

\newcommand{\tay}{ \otimes }
\newcommand{\pinn}{\Pi(z,w)}
\newcommand{\ez}{e(z)}
\newcommand{\ebw}{\bar e(w)}
\newcommand{\ebz}{\bar e(z)}
\newcommand{\ezebw}{\ez\tay\ebw}
\newcommand{\orig}{\pinn\tay\ez\tay\ebw}
\newcommand{\gdzez}{\bigg[g(z)\,dz\tay\ez\bigg]}
\newcommand{\gdwbebw}{\bigg[\overline{g(w)}\,\dwb\tay\ebw\bigg]}

\newcommand{\negsmall}{{\hbox{($<0$ for small $t$ as $t\searrow0$)}}}

\newcommand{\ru}{r_U}
\newcommand{\tu}{\theta_U}

\newcommand{\bigoo}[1]{\bigo{t^{#1}}}

\newcommand{\sidey}{{t^{30}\over\pi^6\det A(t)\det\Lambda(t)}}
\newcommand{\sideshort}{{t^{30}+\bigoo{33}\over\pi^6\det A(t)\det\Lambda(t)}}
\newcommand{\rr}{r_R}
\newcommand{\trr}{\theta_R}

\newcommand{\boxform}[1]{\lower1ex\hbox{\boxit{\hbox{$#1$}}}}
\newcommand{\boxformdisp}[1]{\lower1.5ex\hbox{\boxit{\hbox{$#1$}}}}
\newcommand{\aaa}{{8666449635704832000000}}
\newcommand{\bbb}{{644972544000}}
\newcommand{\ccc}{{8640\sqrt3}}
\newcommand{\ddd}{{27993600}}
\newcommand{\eee}{{161243136000}}
\newcommand{\underx}{\kern-5.2em\lower.5ex\hbox{$\underbrace{\hskip5em}_x$}}
\newcommand{\oversqrtx}{\kern-10.5em\raise4ex\hbox{$\braceld\leaders\vrule\hskip7em\braceru\bracelu\leaders\vrule\hskip.5em\bracerd$}\kern-2.7em\smash{\raise5.5ex\hbox{$\scriptstyle\sqrt{1+x}$}}}
\newcommand{\overgeo}{\kern-9.5em\raise2.5ex\hbox{$\braceld\leaders\vrule\hskip6.2em\braceru\bracelu\leaders\vrule\hskip.8em\bracerd$}\kern-2.5em\smash{\raise4.8ex\hbox{${1\over1+x}$}}\hskip1.5em}
\newcommand{\underxspec}[4]{\kern-#1em\lower#2ex\hbox{$\underbrace{\hskip#3em}_x$}\hskip#4em}
\newcommand{\underarctan}[4]{\kern-#1em\lower#2ex\hbox{$\underbrace{\hskip#3em}_{\arctan x}$}\hskip#4em}
\newcommand{\undercos}[4]{\kern-#1em\lower#2ex\hbox{$\underbrace{\hskip#3em}_{\cos x}$}\hskip#4em}
\newcommand{\overgeospec}[7]{\kern#1em\raise#2ex\hbox{$\braceld\leaders\vrule\hskip#3em\braceru\bracelu\leaders\vrule\hskip#4em\bracerd$}\kern-#5em\smash{\raise#6ex\hbox{${1\over1+x}$}}\hskip#7em}
\newcommand{\undersqrt}[4]{\kern-#1em\lower#2ex\hbox{$\underbrace{\hskip#3em}_{\sqrt{1+x}}$}\hskip#4em}
\newcommand{\elem}[1]{e_#1(\lll1,\lll2,\lll3,\lll4)}
\newcommand{\lefty}{{1\over\pi^6\det A(t)\det\Lambda(t)}}

\newcommand{\leftyyy}{{t^{30}+\bigoo{33}\over\scriptstyle\pi^6\!\det\!A(t)\!\det\!\Lambda(t)\!\prod\!\lll i}}

\newcommand{\bi}[2]{{{O}(t^{\frac#1#2})}}
\newcommand{\omat}[1]{{
\bracketmatrix{
{\cal O}\left(t^{#1}\right)&0&0&0 \cr
0&{\cal O}\left(t^{#1}\right)&0&0 \cr
0&0&{\cal O}\left(t^{#1}\right)&0 \cr
0&0&0&{\cal O}\left(t^{#1}\right)
}
}}
\newcommand{\allomat}[1]{{
\bracketmatrix{
{\cal O}\!\left(t^{#1}\right)&{\cal O}\!\left(t^{#1}\right)&{\cal O}\!\left(t^{#1}\right)&{\cal O}\!\left(t^{#1}\right)\cr
{\cal O}\!\left(t^{#1}\right)&{\cal O}\!\left(t^{#1}\right)&{\cal O}\!\left(t^{#1}\right)&{\cal O}\!\left(t^{#1}\right)\cr
{\cal O}\!\left(t^{#1}\right)&{\cal O}\!\left(t^{#1}\right)&{\cal O}\!\left(t^{#1}\right)&{\cal O}\!\left(t^{#1}\right)\cr
{\cal O}\!\left(t^{#1}\right)&{\cal O}\!\left(t^{#1}\right)&{\cal O}\!\left(t^{#1}\right)&{\cal O}\!\left(t^{#1}\right)\cr
}
}}

\newcommand{\lt}[1]{\tilde\lambda_{#1}}

\newcommand{\genm}[3]{{{}^{#1}\!M_{{#2}{#3}}(t)}}
\newcommand{\genmr}[3]{{{}^{#1}\!M_{{#2}{#3}}(r)}}
\newcommand{\yy}[2]{{Y_{{#1}{#2}}}(t)}
\newcommand{\oaf}[1]{\big[\bigoo{#1}\!\big]_{4\times4}}
\newcommand{\daf}[1]{\mathop{\big[\bigoo{#1}\!\big]_{4\times4}}_{\kern-1.3em{\rm (diag)}}}
\newcommand{\ut}{\tilde U(t)}
\newcommand{\al}[1]{{\alpha_{#1}}}
\newcommand{\be}[1]{{\beta_{#1}}}
\newcommand{\ga}[1]{{\gamma_{#1}}}
\newcommand{\de}[1]{{\delta_{#1}}}

\newcommand{\doc}{paper}
\newcommand{\crit}{\mathop{\rm Crit}^{\nablauh}(s)}
\newcommand{\Crit}{\crit}

\renewcommand{\phi}{\varphi}

\newcommand{\knnk}{{K_{nk}^N}}
\newcommand{\knkm}{{K_{nkm}^\infty}}
\newcommand{\sume}[1]{{\sum_{jk}\!{}^{#1}\!\varepsilon_{jk}(t)w_j\wb_k}}
\newcommand{\sumew}[1]{{\sum_{jk}\!{}^{#1}\!\varepsilon_{jk\ell m}(t)w_j\wb_kw_{\ell}\wb_m}}
\renewcommand{\SS}[1]{\Sigma_{#1}(w,t)}
\newcommand{\eps}[3]{{}^{#1}\!\varepsilon_{#2#3}(t)}
\newcommand{\edub}{\,e^{-\angleenclose{w,w}}\,dw}
\newcommand{\pound}{\mathop{\char "23}\nolimits}
\newcommand{\Vol}[1]{{\mathop{\rm Vol}\nolimits\left[#1\right]}}
\newcommand{\jpd}{\mathop{\rm JPD}}

\newcommand{\kinf}[3]{{K_{#1#2#3}^\infty}}

\newcommand{\kN}[3]{{K_{#1#2#3}^N}}
\newcommand{\kn}[2]{{K_{#1#2}^N}}
\newcommand{\npower}{{{}^2}}

\newcommand{\etothelor}{e^{\thelor}}

\newcommand{\dist}{\mathop{\rm dist}}
\newcommand{\tm}{T_M}
\newcommand{\tmholo}{T'_M}
\newcommand{\tmantiholo}{T''_M}
\newcommand{\tsm}{T^*_M}
\newcommand{\tsmholo}{T^{*\prime}_M}
\newcommand{\tsmantiholo}{T^{*\prime\prime}_M}
\newcommand{\tpm}{\bigwedge^p T^*_M}
\newcommand{\tpqsm}{T^{*(p,q)}_M}

\newcommand{\ddzone}{{\partial\,\over\partial z_1}}
\newcommand{\ddzn}{{\partial\,\over\partial z_n}}
\newcommand{\ddzbone}{{\partial\,\over\partial \bar z_1}}
\newcommand{\ddzbn}{{\partial\,\over\partial \bar z_n}}
\newcommand{\spa}{\mathop{\rm span}}
\newcommand{\forms}[1]{{\cal A}^{#1}(L)}
\newcommand{\ot}{\otimes}
\newcommand{\mul}{{\cal L}}
\newcommand{\volm}{{\rm Vol}_M}

\newcommand{\pih}{\Pi_1^{\bf H}}
\newcommand{\pit}{\tilde\Pi_N}

\newcommand{\nablaHz}{\nabla_{\!z}^{\bf H}}
\newcommand{\nablaHzh}{\left(\nabla_z^{\bf H}\right)'}

\newcommand{\nabHzh}{\nablaHzh}
\newcommand{\nabHwy}{\left(\nabla_w^{\bf H}\right)''}
\newcommand{\sppppp}{{\pih((z,0),(w,0))}|_{(\zeta_p,\zeta_{p'})}}
\newcommand{\spppppp}{{\Pi_N\left(\zovern,\wovern\right)}|_{(\zeta_p,\zeta_{p'})}}
\newcommand{\zovern}{{z\over\sqrt N}}
\newcommand{\wovern}{{w\over\sqrt N}}
\newcommand{\nablahn}{{\nabla_{\!\!h^N}}}
\newcommand{\ddzop}{\ddz{\phantom x}}
\newcommand{\lamH}{\Lambda^{\bf H}}
\newcommand{\deltaH}{\Delta^{\bf H}}
\newcommand{\hred}{{\bf H}_{\rm red}}
\newcommand{\angleenclose}[1]{\langle#1\rangle}
\newcommand{\ddz}[1]{{\del #1\over\del z}}

\newcommand{\ddzjf}{{\del f\over\del z_j}}
\newcommand{\ddzbjf}{{\del f\over\del \bar z_j}}

\newcommand{\ddzdz}[1]{{\del^2 #1\over\del z^2}}

\newcommand{\ddwbdwb}[1]{{\del^2 #1\over\del\bar w^2}}
\newcommand{\ddzdwbdwb}[1]{{\del^3 #1\over\del z\del\bar w^2}}

\newcommand{\ddzdwb}[1]{{\del^2 #1\over\del z\del\bar w}}

\newcommand{\dwb}{d\bar w}

\newcommand{\ddwb}[1]{{\del #1\over\del\bar w}}

\newcommand{\nabz}{\nabla_z}

\newcommand{\nabh}{\nabla'}
\newcommand{\naby}{\nabla''}
\newcommand{\nabhz}{\nabla_{\!\!z}'}
\newcommand{\nabyz}{\nabla_{\!\!z}''}
\newcommand{\nabhw}{\nabla_w'}
\newcommand{\nabyw}{\nabla_{\!\!w}''}

\newcommand{\tts}{\angleenclose{\;\,,\;}}
\newcommand{\surj}{\twoheadrightarrow}
\newcommand{\iso}{\cong}
\newcommand{\clap}[1]{\hbox to 0pt{\hss#1\hss}}

\newcommand{\mathrlap}{\mathpalette\mathrlapinternal}

\newcommand{\mathrlapinternal}[2]{ \rlap{$\mathsurround=0pt#1{#2}$}}

\newcommand{\as}[2]{\mathbin{{\displaystyle\mathop{\hbox to #2pt{$\mathord-\mkern-6mu\cleaders\hbox{$\mkern-2mu\mathord-\mkern-2mu$}\hfill\mkern-6mu\mathord\rightarrow$}}^{#1}}}}
\newcommand{\brac}[2]{{
{
   t^\frac52 u_{#1 #2}
    \over
   \sqrt{\lll{#1}}
   }
}}
\newcount\mc
\def\mn{{{\the\mc}}}
\newcount\ac
\def\an{{{\the\ac}}}

\journalname{}

\begin{document}

\title{Scaled Correlations of Critical Points of Random Sections on Riemann Surfaces}

\titlerunning{Scaled Correlations}        

\author{John Baber}

\institute{John Baber\at
              University of Connecticut\\
              Department of Mathematics\\
              196 Auditorium Rd.\\
              Storrs, CT 06269\\
              Tel.: +1-860-486-1288\\
              \email{baber@math.uconn.edu}
}

\date{Received: date / Accepted: date}

\maketitle

\begin{abstract}
In this paper we prove that as $N$ goes to infinity, the scaling limit of the correlation between critical points $z_1$ and $z_2$ of random holomorphic sections of the $N$-th power of a positive line bundle over a compact Riemann surface tends to $2/(3\pi^2)$ for small $\sqrt{N}|z_1-z_2|$. The scaling limit is directly calculated using a general form of the Kac-Rice formula and formulas and theorems of Pavel Bleher, Bernard Shiffman, and Steve Zelditch. 
\end{abstract}

\section{Introduction}
\label{intro}

This \doc\ studies the behavior of the critical points of gaussian random holomorphic sections of the $N$-th power of a holomorphic line bundle $L$ on a Riemann surface $M$ as $N\to\infty$, as is studied in \cite{cpsv1}, \cite{cpsv2}, and \cite{cpsv3}.  In the particular case where $L={\cal O}(1)$, the so-called hyperplane section bundle over $M=\mathbb{C}\mathbb{P}^1$, sections of $L^N$ correspond to homogeneous polynomials of degree $N$, the $SU_2$ polynomials, so the results in this \doc\ apply to the critical points of random polynomials $\sum\sqrt{n\choose k}c_kz^k$ with $c_k$ identically distributed gaussian random variables.  In this way, this \doc\ examines one small facet of the theory of random polynomials and random holomorphic functions.

Since what may have been the first study of critical points of random curves in \cite{riceold}, this area of research has led to results of interest in mathematics, probability theory, and physics.  For instance, the classical result of Hammersley in \cite{hammersley} that for $f(z):=\sum_{j=0}^Nc_jz^j$ with $c_j$ independent standard gaussian random variables, as $N\to\infty$, the complex zeroes tend toward the unit circle in $\mathbb{C}$ and its generalization by Bloom and Shiffman in \cite{bs} (also discussed in \cite{bloom}), namely that as $N\to\infty$, the common zeroes of $m$ random polynomials $f_k(z):=\sum_{|J|\leq N}c_J^kz_1^{j_1}\cdots z_m^{j_m}$ in $\mathbb{C}^m$ are concentrated near the ``distinguished boundary'' of the $m$-dimensional polydisc.  Since the zeroes of a collection of $m$ polynomials in $m$ variables is almost surely discrete, for random $f_i$, the set $\{f_1(z)=f_2(z)=\dots=f_m(z)=0\}$ is a random point process on $\mathbb{C}^m$ of interest in probability theory.

How much should zeroes and critical points of random polynomials or random holomorphic functions $\sum_{k=0}^\infty{a_k\over k!}z^k$ be expected to vary from their expected behavior?  This type of question is addressed in \cite{sodin4}, \cite{sodin1}, \cite{sodin2}, \cite{sodin3}.  This \doc\ examines how pairs of critical points are correlated by examining the 2-point correlation function, $K_2(z,w)$.

The main theorem of this \doc\ says that the scaling limit of the correlation between critical points of random holomorphic sections of the $N$-th power of a positive line bundle over a compact Riemann surface tends to ${2\over3\pi^2}$ as $N\to\infty$ for small $r:=\sqrt N|\zeta_1-\zeta_2|$. i.e.
\begin{theorem}
\label{main}
For any positive hermitian line bundle $L$ over any compact Riemann surface $M$
\begin{equation}
\label{eq:question2}
\lim_{N\to\infty}{1\over N\npower}\kN211\left({\zeta_1\over\sqrt N},{\zeta_2\over\sqrt N}\right)={2\over3\pi^2}+\bigo{r^2} 
\end{equation}
where $ r:=\dist(\zeta_1,\zeta_2) $ i.e. as the distance between critical points gets smaller, their scaled limit correlation approaches $\displaystyle{2\over3\pi^2}$ uniformly in $\zeta_1,\zeta_2$.
\end{theorem}


$\kN211(z,w)$ is calculated via the generalized form of the Kac-Rice formula of \cite{kac},\cite{rice}
\begin{equation}
K(t)=\int|\xi|\jpd(0,\xi;\,t)\,d\xi
\end{equation}
where $\jpd(x,\xi;\,t)$ denotes the joint probability distribution of $x=f(t)$ and $\xi=f'(t)$.

Though we know no immediate interpretation of the constant ${2\over3\pi^2}$, the fact that it is not 0 is interesting.  This contrasts with the fact that the scaling limit correlation of \it zeroes\rm\ of random sections on a compact Riemann surface is $\bigo{r^2}$ as was proved in general in \cite{pl} and \cite{bsz} and specfically for $\mathbb{C}\mathbb{P}^1$ in \cite{hannay}.

This paper is based on the thesis submitted to the Department of Mathematics at Johns Hopkins University in 2010 which was read by Bernard Shiffman (Advisor) and Steve Zelditch.

The introductions of \cite{cpsv1}, \cite{cpsv2}, and \cite{bsz} give a description of the basic objects of study and the physical motivation for them.  The next few sections summarize the more thorough descriptions given there.

\section{Notation and Formulas}

Throughout these definitions, $M$ will denote a complex manifold of complex dimension $n$ with complex coordinates $(z_1,\ldots,z_n)$.  $M$ will also be thought of as a $2n$-dimensional real manifold with coordinates $(x_1,y_1,\ldots,x_n,y_n)$ where $z_j=:x_j+iy_j$.  $\mul$ denotes Lebesgue measure on $\mathbb{C}$ and ${\cal B}$ denotes the borel subsets of $\mathbb{C}$.  In general $L$ will be a holomorphic line bundle over $M$.  For standard results and definitions about line bundles, see Chapter 1 of \cite{gh} for instance.

We begin by summarizing our notation. $\tm$ denotes the set of smooth complex-valued vectors on $M$. i.e. $\tm\!_{,p}$ is the space of $\mathbb{C}$-linear derivations in the ring of complex-valued ${\cal C}^\infty$ functions on $M$ near $p$.
\begin{equation}\tm=\tmholo\oplus\tmantiholo\end{equation}
where
\begin{align}
  \tmholo&:=\spa\left(\ddzone,\ldots,\ddzn\right)\qquad\hbox{the ``holomorphic'' tangent space} \\
  \tmantiholo&:=\spa\left(\ddzbone,\ldots,\ddzbn\right)\qquad\hbox{the ``antiholomorphic'' tangent space}
\end{align}
$\tsm$ denotes the dual space of $\tm$, i.e. the set of smooth complex-valued covectors or 1-forms on $M$.  $$\tsm=\tsmholo\oplus\tsmantiholo$$ where
\begin{align}
  \tsmholo&:=\spa\left(dz_1,\ldots,dz_n\right)\qquad\hbox{the ``holomorphic'' cotangent space} \\
  \tsmantiholo&:=\spa\left(\dzb_1,\ldots,\dzb_n\right)\qquad\hbox{the ``antiholomorphic'' cotangent space}
\end{align}
Recall that
\begin{equation}
d=\del+\delb
\end{equation}
where
\begin{equation}
df=\underbrace{\sum_{j=1}^n\ddzjf\,dz_j}_{\del f}+\underbrace{\sum_{j=1}^n\ddzbjf\,\dzb_j}_{\delb f}
\end{equation}

We let ${\cal A}^p(L)$ denote the sheaf of smooth $L$-valued $p$-forms.  i.e. for any open $U\subset M$, ${\cal A}^p(L)(U):=$\hfill
\begin{equation}
\{\left.\strut\omega\right|_x\otimes e_U(x)\;\mid\;\omega\in \left.\strut \tpm\right|_U\hbox{and $e_U$ a local frame above $U$}\}.
\end{equation}
We also let ${\cal A}^{p,q}(L)$ denotes the sheaf of smooth $L$-valued $(p,q)$-forms.  i.e. for any open $U\subset M$, ${\cal A}^{p,q}(L)(U):=$\hfill
\begin{equation}\{\left.\strut\omega\right|_x\otimes e_U(x)\;\mid\;\omega\in \left.\strut \tpqsm\right|_U\hbox{and $e_U$ a local frame above $U$}\}\end{equation}

The $\delb$ operator is extended to act on sections via
\[
\begin{matrix}
  \delb:&\forms{p,q}&\longrightarrow&\forms{p,q+1} \\
  &\omega\ot e&\longmapsto&\delb\omega\ot e
\end{matrix}
\]

Using the $\tsm=\tsmholo\oplus\tsmantiholo$ decomposition, we write $\nabla=\nabla'+\nabla''$ for any connection $\nabla$ where
\begin{align}
  \nabla':\forms0\to\forms{1,0} \\
  \nabla'':\forms0\to\forms{0,1}
\end{align}

Given a hermitian line bundle $(L,\tts_h)\to M$, the Chern connection associated to $\tts_h$, will be written $\nablauh$.
When the $h$ is obvious, $\nablauh$ and $\tts_h$ will just be written $\nabla$ and $\tts$.

For given $N$, we will choose local coordinates on $U$ and a local frame $e^N_U$ for $L^N$ over $U\subset M$ such that $h(z)=1-|z|^2+\bigo{|z|^3}$ by taking an arbitrary frame and multiplying by a smooth function with appropriate first and second order terms.  Then
\begin{equation}
\nablahn=d+N\del\log h
\end{equation}
and
\begin{align}
\ddzop\log h&=\ddzop\log(1-|z|^2+\bigo{|z|^3}) \notag \\
&={1\over1-|z|^2+\bigo{|z|^3}}(-\bar z+\bigo{|z|^2}) \\
&=(1+\bigo{|z|^2})(-\bar z+\bigo{|z|^2}) \notag \\
&=-\bar z+\bigo{|z|^2} \notag
\end{align}
Also
\begin{equation}
\nabla_h''=\delb
\end{equation}

The curvature form for $\nablauh$ will be written $\Theta_h$.  Note that because $L$ is a \it line\rm\ bundle, $\Theta_h$ is just the $1\times1$ matrix $\bracketmatrix{\delb\theta}=\bracketmatrix{\delb\del\log h}$.

We now summarize notation from probability theory.  For a random variable 
\begin{equation}
X:(\Omega,\Sigma,P)\to(\mathbb{R},{\rm BorelSets},d{\rm Lebesgue})
\end{equation}
we'll write the cumulative distribution function of $X$ as
\begin{equation}
F_X(t):=\P{X^{-1}((-\infty,t])}
\end{equation}
When $X\in L^1(P)$ we denote the expected value of $X$ by
\begin{equation}
\E{X}:=\int_\Omega X\,dP
\end{equation}
Note when $X$ happens to have a probability density function $f_X$,
\begin{equation}
\E{X}=\intr tf_X(t)\,dt
\end{equation}

\begin{definition}\label{cgauss}A \it centered complex gaussian\rm\ random variable is a random variable
\begin{equation}
X:(\Omega,\Sigma,P)\to(\mathbb{C},{\cal B},\mul)
\end{equation}
whose distribution is 
\begin{align}
(X_*P)(B):=\P{X^{-1}(B)}=\int_B{1\over\pi\sigma^2}e^{-{1\over\sigma^2}|z|^2}\,d\mul(z)
\end{align}
When $\sigma=1$ we say $X$ is a \it standard complex gaussian\rm.
\end{definition}

Note any centered gaussian has expected value 0:
\begin{equation}
\int_\Omega X\,dP=\intc{z\over\pi\sigma^2}e^{-{1\over\sigma^2}|z|^2}\,d\mul(z)=0
\end{equation}

\begin{definition}\label{jointgauss}More generally a collection of random variables $X_j:\Omega\to\mathbb{C}$ is said to be \it jointly gaussian\rm\ if the complex valued random variable
\begin{equation}
a_1X_1+\cdots+a_nX_n
\end{equation}
is a centered complex gaussian for any $a_j\in\mathbb{C}$.
\end{definition}

\begin{definition}\label{covmatrix}  The $n\times n$ symmetric positive semi-definite matrix
\begin{equation}
\Delta:=\bracketmatrix{\E{X_i\bar X_j}}_{i,j=1\ldots n}
\end{equation}
is called the \it covariance matrix\rm\ of $\vec X$.  When the $X_i$ are linearly independent, as in our calculation, $\Delta$ is non-singular, i.e. \it positive definite\rm.
\end{definition}

When the $X_i$ are linearly independent, Definition \ref{jointgauss} is equivalent to a more probability density style description.  Specifically, Definition \ref{jointgauss} in this case is equivalent to demanding that the random vector
\begin{equation}
\bracketmatrix{X_1\cr\vdots{}\cr\strut X_n}:\Omega^n\to\mathbb{C}^n
\end{equation}
has distribution
\begin{equation}
(\vec X_*P)(B)=\P{\vec X^{-1}(B)}=\int_B{1\over\pi^n\det\Delta}e^{-\angleenclose{\Delta^{-1}z,z}}\,dz
\end{equation}
where
\begin{equation}
dz=d\mul(z_1)\w\ldots\w d\mul(z_n)
\end{equation}

The distribution $\vec X_*P$ for any $P$ as above is called the \it joint probability distribution\rm\ of the $X_j$.

\begin{lemma}\label{jointgausspreserve} If $X_1,\ldots,X_n$ are jointly gaussian then the entries of $L(\vec X)$ are, too, for any linear surjection $L:\mathbb{C}^n\surj\mathbb{C}^m$
\end{lemma}
\begin{proof}
If $\bracketmatrix L=\bracketmatrix{\l_{ij}}$ and
\begin{equation}
L\left(\bracketmatrix{X_1\cr\vdots\cr X_n}\right)=\bracketmatrix{Y_1\cr\vdots\cr Y_m}
\end{equation}
then $a_1Y_1+\cdots+a_mY_m=$\begin{align*}
  &a_1(\l_{11}X_1+\cdots+\l_{1n}X_n)+\cdots+a_m(\l_{m1}X_1+\cdots+\l_{mn}X_n) \\
  =&(a_1\l_{11}+\cdots+a_m\l_{m1})X_1+\cdots+(a_1\l_{1n}+\cdots+a_m\l_{mn})X_n
\end{align*}
a centered complex gaussian for any $a_j\in\mathbb{C}$.\hfill\qed
\end{proof}

\section{Random Sections and the Two Point Kernel}

Here we define what we mean by ``random sections'' of the bundle $L^N\to M$.
\begin{definition}\label{htothen} The metric $h$ induces hermitian metrics $h^N$ on $L^N$ given by $h^N(z):=h(z)^N$ i.e.
$\angleenclose{s_1\ot\cdots\ot s_N,t_1\ot\cdots\ot t_N}_{L^N}=$
\begin{align*}
  &\angleenclose{f_1e_U\ot\cdots\ot f_Ne_U,g_1e_U\ot\cdots\ot g_Ne_U}_{L^N} \cr
=&f_1\cdots f_N\bar g_1\cdots\bar g_N\angleenclose{e_U\ot\cdots\ot e_U,e_U\ot\cdots\ot e_U}_{L^N} \cr
:=&f_1\cdots f_N\bar g_1\cdots\bar g_N\angleenclose{e_U,e_U}_L^N=f_1\cdots f_N\bar g_1\cdots\bar g_N h(z)^N
\end{align*}
\end{definition}

\begin{definition}\label{weirdip} Using $h^N$ we can create a new hermitian inner-product on $H^0(M,L^N)$ by
\begin{equation}
\angleenclose{s,t}=\int_M\angleenclose{s,t}_{L^N}\,d\volm\qquad s,t\in H^0(M,L^N)
\end{equation}
\end{definition}

Throughout the rest of this section, $(s_j^N)$ will denote an orthonormal basis for $H^0(M,L^N)$ with respect to $\tts$.

\begin{definition}\label{randomsection}
We can define a gaussian probability measure $P$ on\break $H^0(M,L^N)$.  Given
\begin{equation}
H^0(M,L^N)\ni s=\sum_{j=1}^\l c_j(s)\,s_j^N
\end{equation}
for any borel collection of sections ${\cal S}$,
\begin{equation}
\P{{\cal S}}:=\int_{\cal S}{1\over\pi^n}e^{-\angleenclose{c_j(s),c_j(s)}}\,dc(s)
\end{equation}
where $dc(s)$ is $2\l$-dimensional Lebesgue measure.
\end{definition}

$P$ is characterized by the property that the $2\l$ real variables $\realpart(c_j)$ and $\Im(c_j)$ are independent random variables with mean 0 and variance $\fracs12$.  Specifically
\begin{equation}
\E{c_j}=0\qquad\E{c_jc_k}=0\qquad\E{c_j\bar c_k}=\delta_{jk}
\end{equation}

For $c_1,\ldots,c_\l$ jointly gaussian, consider the random holomorphic section
\begin{equation}
s(z):=\sum_{j=1}^\l c_js_j^N(z)
\end{equation}
and the map

\newcommand{\ay}{{(\nabhz s)(z)}}
\newcommand{\bee}{{(\nabhz\nabhz s)(z)}}
\newcommand{\ce}{{(\nabyz\nabhz s)(z)}}
\newcommand{\dee}{{(\nabhw s)(w)}}
\renewcommand{\ee}{{(\nabhw\nabhw s)(w)}}
\newcommand{\ef}{{(\nabyw\nabhw s)(w)}}
\begin{equation}
\bracketmatrix{c_1\cr\vdots\cr c_\l}\mathop{\longmapsto}^\lambda
\bracketmatrix{
\ay\cr
\bee\cr
\ce\cr
\dee\cr
\ee\cr
\ef\cr
}
=
\bracketmatrix{
\sum_{j=1}^\l c_j(\nabhz s_j^N)(z)\cr
\sum_{j=1}^\l c_j(\nabhz\nabhz s_j^N)(z)\cr
\sum_{j=1}^\l c_j(\nabyz\nabhz s_j^N)(z)\cr
\sum_{j=1}^\l c_j(\nabhw s_j^N)(w)\cr
\sum_{j=1}^\l c_j(\nabhw\nabhw s_j^N)(w)\cr
\sum_{j=1}^\l c_j(\nabyw\nabhw s_j^N)(w)\cr
}
\end{equation}
For fixed $z$ and $w$, $\lambda$ is a linear map so Lemma \ref{jointgausspreserve} says the entries of $\lambda(\vec c)$ are jointly gaussian.  Their covariance matrix is
\begin{equation}
  \Delta=\bracketmatrix{\Delta_{jk}}_{j,k=1}^{6}
\end{equation} where
\begin{equation}
\label{delta}
\Delta_{jk}=\E{\lambda_j\ot\overline{\lambda_k}}
\end{equation}
and
\begin{align}
  &\lambda_1=\ay \notag\\
  &\lambda_2=\bee \notag\\
  &\lambda_3=\ce \\
  &\lambda_4=\dee \notag\\
  &\lambda_5=\ee \notag\\
  &\lambda_6=\ef \notag
\end{align}
by abuse of notation.  Each entry should be replaced by its coefficient when written in a local frame about $z$ and $w$.

In fact all of the entries of $\Delta$ used in our calculation can be rewritten in terms of derivatives of an important invariant of $P$ called the ``two point kernel''.
\begin{definition}\label{twopointkernel} The \it two-point kernel\rm\ (or \it covariance kernel\rm) for $H^0(M,L^N)$ is defined by
\begin{equation}
\Pi_N(z,w):=\sum_{j=1}^\l s_j^N(z)\otimes\overline{s_j^N(w)}\in L^N_z\otimes\overline{L^N_w}\quad(z,w)\in M
\end{equation}
\end{definition}
Since $L^N$ is hermitian, $\Pi_N$ is the Szeg\H o kernel of $(L^N,h^N)$, i.e. the orthogonal projection 
\begin{equation}
\Pi_{N,h^N,\volm}:L^2(M,L^N)\to H^0(M,L^N)
\end{equation}
with respect to $\tts$.

$\Pi_N$ and the entries of $\Delta$ are related because
\begin{align}
  \E{s(z)\ot\overline{s(w)}}&=\E{\sum_{j=1}^\l c_js_j^N(z)\ot\overline{\sum_{k=1}^\l c_ks_k^N(w)}} \cr
  &=\sum_{j,k=1}^\l\E{c_j\bar c_k}s_j^N(z)\ot\overline{s_k^N(w)} \\
  &=\sum_{j,k=1}^\l\delta_{jk}s_j^N(z)\ot\overline{s_k^N(w)} =\sum_{j=1}^\l s_j^N(z)\ot\overline{s_j^N(w)} \notag\\
  &=:\Pi_N(z,w) \notag
\end{align}
so differentiating both sides yields, for instance,
\newcommand{\szz}{{\E{s(z)\ot\overline{s(z)}}}}
\newcommand{\szw}{{\E{s(z)\ot\overline{s(w)}}}}
\newcommand{\swz}{{\E{s(w)\ot\overline{s(z)}}}}
\newcommand{\sww}{{\E{s(w)\ot\overline{s(w)}}}}
\begin{align}
  \nabhz\nabyw\pizw&=\nabhz\nabyw\szw=\nabhz\E{s(z)\ot\nabyw\overline{s(w)}}\cr
  &=\nabhz\E{s(z)\ot\overline{\nabhw s(w)}}=\E{\nabhz s(z)\ot\overline{\nabhw s(w)}}\cr
  &=\Delta_{14}
\end{align}

\section{The Kac-Rice Theorem}
\def\fa{f_{\vec a}}
Various generalizations of Rice's original theorem \cite{riceold}(3) are referred to as ``The Kac-Rice Theorem'' in current literature.  What is meant by ``using the Kac-Rice theorem'' is that the expected density of zeroes of a random linear combination of functions
\begin{equation}
\fa(x):=\sum_{j=0}^\l a_jf_j(x)
\end{equation}
is found by integrating the joint distribution of $\fa$ and $\fa'$ with $\fa$ replaced by $0$ against $\|\fa'\|$.
For instance, the single-variable real Kac-Rice theorem says the following:\hfil\break
Take $P$ a probability measure on $\mathbb{R}^\l$ and $f_1,\ldots,f_\l$ a collection of analytic functions.  For fixed $t$,
\begin{equation}
x_t:=\fa(t):\mathbb{R}^\l\to\mathbb{R}
\end{equation}
and
\begin{equation}
\xi_t:=\fa'(t):\mathbb{R}^\l\to\mathbb{R}
\end{equation}
are random variables so they have a joint probability distribution, namely the distribution for the random variable
\begin{equation}
X_t:=\bracketmatrix{x_t\cr\xi_t}:\mathbb{R}^\l\to\mathbb{R}^2
\end{equation}
specifically
\begin{equation}
(X_t)_*P=:D_t(x,\xi)
\end{equation}
\begin{definition}\label{zerofunction} For the function $\fa$ define the measure
\begin{equation}
Z_{\fa}:=\sum_{\fa(t_j)=0}\delta_{t_j}
\end{equation}
\end{definition}
The demand that the $f_j$ be analytic ensures $\fa$ has only finitely many zeroes on bounded intervals, so $Z_{\fa}$ is a sum of point masses.  $Z_{\fa}$ can be generalized to mean the current of integration along the regular points of the variety $\{\fa=0\}$ even when not discrete, but that is not necessary for our computation.

The Kac-Rice theorem says
\begin{equation}
\E{Z_{\fa}}=K(t)\,dt
\end{equation}
where
\begin{equation}
K(t):=\intr D_t(0,\xi)|\xi|\,d\xi
\end{equation}
\begin{definition}\label{onepointcorr} The $K(t)\,dt$ above is called the \it one-point correlation\rm\ or \it one-point density\rm\ of $Z_{\fa}$.\end{definition}
\begin{definition}\label{npointcorr} If we define the measure of $n$ simultaneous zeroes\end{definition}
\begin{equation}
  Z_{\fa}^n:=\sum_{\{\vec t\in M\times\cdots\times M\mid\fa(t_1)=\cdots=\fa(t_n)=0\}}\delta_{\vec t}
\end{equation}
the Kac-Rice theorem says the same for a measure $K_n(\vec t)\,d\vec t$ called the \it$n$-point correlation\rm\ or \it$n$-point density\rm\ of $Z_{\fa}^n$.

This \doc\ is, in fact, concerned with the 2-point correlation of the simultaneous zeroes of two random \it sections\rm\ $\nabla s^N(z),\nabla s^N(w)\in\tsm\ot L^N$ for $s^N\in H^0(M,L^N)$ which are called \it critical points\rm\ of $s^N$.  Here ``random'' means that the $s^N$ are chosen with respect to the gaussian probability measure on $H^0(M,L^N)$ given in Definition \ref{randomsection} and $\dim_{\mathbb{C}}M=1$.  In this particular case, the Kac-Rice theorem is
\begin{equation}
\E{Z_{\nabla s(z),\nabla s(w)}}=K_2(z,w)\,dz\,dw
\end{equation}
where
\begin{align}
  K_2(z,w)&=\int_{W}D_{z,w}(\vec0,\vec\xi)\det(\xi^1(\xi^1)^*)^\frac12\det(\xi^2(\xi^2)^*)^\frac12\,d\xi \\
  &=\int_{W}D_{z,w}(\vec0,\vec\xi)\cdot|\det\xi^1|\cdot|\det(\xi^2)|\,\,d\xi
\end{align}
due to the makeup of $\xi^1$ and $\xi^2$.
Here $D_{z,w}(\vec x,\vec\xi)$ is the joint probability distribution of the random sections
\begin{align}
x_1&=\nabla s(z)=\nabhz s(z)\qquad\hbox{since $s$ is holomorphic} \notag\\
x_2&=\nabla s(w)=\nabhw s(w)\qquad\hbox{since $s$ is holomorphic}  \notag\\
\xi_1&=\nabh\nabla s(z)  \notag\\
\xi_2&=\naby\nabla s(z)  \\
\xi_3&=\nabh\nabla s(w)  \notag\\
\xi_4&=\naby\nabla s(w)  \notag
\end{align}
and
\begin{align*}
\xi^1&=``\nabla\nabla s(z)" =\bracketmatrix{\nabh\nabla s(z)&\overline{\naby\nabla s(z)}\cr\naby\nabla s(z)&\overline{\nabh\nabla s(z)}}=\bracketmatrix{\nabhz\nabhz s(z)&\overline{\nabyz\nabhz s(z)}\cr\nabyz\nabhz s(z)&\overline{\nabhz\nabhz s(z)}} \cr
\xi^2&=``\nabla\nabla s(w)" =\bracketmatrix{\nabh\nabla s(w)&\overline{\naby\nabla s(w)}\cr\naby\nabla s(w)&\overline{\nabh\nabla s(w)}}=\bracketmatrix{\nabhw\nabhw s(w)&\overline{\nabyw\nabhw s(w)}\cr\nabyw\nabhw s(w)&\overline{\nabhw\nabhw s(w)}} \cr
\end{align*}
as in \cite{symplectic}(33).  Here
\newcommand{\wone}{\nabh\!\!\left(\left(\tsm\ot L^N\right)_z\right)}
\newcommand{\wtwo}{\naby\!\!\left(\left(\tsm\ot L^N\right)_z\right)}
\newcommand{\wthree}{\nabh\!\!\left(\left(\tsm\ot L^N\right)_w\right)}
\newcommand{\wfour}{\naby\!\!\left(\left(\tsm\ot L^N\right)_w\right)}
\begin{align*}
&W=\cr
&\wone\!\!\times\!\!\wtwo\!\!\times\!\!\wthree\!\!\times\!\!\wfour
\end{align*}
and $d\xi$ means Lebesgue measure with respect to the hermitian metric on $W$.

As in the previous section, $x_1,x_2,\xi_1,\xi_2,\xi_3$, and $\xi_4$ are jointly gaussian with covariance matrix $\Delta=$(\ref{delta}) so
\begin{equation*}
D_{z,w}(x_1,x_2,\xi_1,\xi_2,\xi_3,\xi_4)={1\over\pi^6\det\Delta(z,w)}\expy{-\left\langle\Delta^{-1}(z,w)\bracketmatrix{\vec x\cr\vec\xi},\bracketmatrix{\vec x\cr\vec\xi}\right\rangle}
\end{equation*}
so
\begin{equation*}
D_{z,w}(0,0,\xi_1,\xi_2,\xi_3,\xi_4)={1\over\pi^6\det\Delta(z,w)}\expy{-\left\langle\Delta^{-1}(z,w)\bracketmatrix{\vec 0\cr\vec\xi},\bracketmatrix{\vec 0\cr\vec\xi}\right\rangle}
\end{equation*}

Dividing $\Delta$ into blocks
\begin{equation}
\label{thesemats}\Delta=\bracketmatrix{\bracketmatrix{A}_{2\times2}&\bracketmatrix{B}_{2\times4}\cr\bracketmatrix{B^*}_{4\times2}&\bracketmatrix{C}_{4\times4}}_{6\times6}
\end{equation}
and using the formula for inverting matrices presented in blocks
\begin{align}
  \Delta^{-1}&=\bracketmatrix{I&-A^{-1}B\cr0&I}\bracketmatrix{A^{-1}&0\cr0&\smash{\underbrace{(C-B^*A^{-1}B)^{-1}}_{\Lambda\mathrlap{{}^{-1}}}}}\bracketmatrix{I&0\cr-B^*A^{-1}&I} \cr
\noalign{\bigskip}
  &=\bracketmatrix{\bigg[A^{-1}+A^{-1}B\Lambda^{-1}B^*A^{-1}\bigg]_{2\times2}&\bigg[-A^{-1}B\Lambda^{-1}\bigg]_{2\times4}\cr\bigg[-\Lambda^{-1}B^*A^{-1}\bigg]_{4\times2}&\bigg[\Lambda^{-1}\bigg]_{4\times4}}
\end{align}
meaning
\begin{align*}
\Delta^{-1}
\bracketmatrix{
  \\
  \begin{bmatrix}0\\0\end{bmatrix}\\
  \scriptstyle2\times1\\
  \\
  \bracketmatrix{\cr \vec\xi \cr\cr}\cr\scriptstyle4\times1
}&= 
\bracketmatrix{\bigg[A^{-1}+A^{-1}B\Lambda^{-1}B^*A^{-1}\bigg]&\bigg[-A^{-1}B\Lambda^{-1}\bigg]\cr\bigg[-\Lambda^{-1}B^*A^{-1}\bigg]&\bigg[\Lambda^{-1}\bigg]}
\bracketmatrix{\cr\bracketmatrix{0\cr0}\cr\cr\bracketmatrix{\cr \vec\xi\cr\cr}\cr\cr}\cr
&=\bracketmatrix{-A^{-1}B\Lambda^{-1} \vec\xi\cr\Lambda^{-1} \vec\xi}
\end{align*}
so
\newcommand{\combovecxi}{\bracketmatrix{\vec 0\\ \vec\xi}}
\begin{align*}
\left[\left\langle\tilde\Delta^{-1}\combovecxi,\combovecxi\right\rangle\right]&=\big[(-A^{-1}B\Lambda^{-1} \vec\xi)^*\quad(\Lambda^{-1} \vec\xi)^*\big]\bracketmatrix{\cr\bracketmatrix{0\cr0}\cr\cr\bracketmatrix{\cr \vec\xi\cr\cr}\cr\cr}=\big[(\Lambda^{-1} \vec\xi)^* \vec\xi\big]\cr
&=\big[\angleenclose{\Lambda^{-1} \vec\xi, \vec\xi}\big]
\end{align*}
Using the formula for determinants of matrices presented as blocks says
\begin{equation}
\label{splitdelta}
\det\Delta=(\det A)(\det\Lambda)
\end{equation}
so the integrand in $K_2(z,w)$ is
\begin{equation}
{e^{-\angleenclose{\Lambda^{-1}(z,w)\vec\xi,\vec\xi}}\over\pi^6\det A(z,w)\det\Lambda(z,w)}\det(\xi^1(\xi^1)^*)^\frac12\det(\xi^2(\xi^2)^*)^\frac12\,d\xi
\end{equation}
Now, as mentioned above, the particular entries of the $\xi^j$ allow the\break$\det(\xi^j(\xi^j)^*)^\frac12$ to be simplified as in \cite{cpsv1}(34).
\begin{align*}
\det(\xi^1(\xi^1)^*)^\frac12&\det(\xi^2(\xi^2)^*)^\frac12\cr
&=\det\left(\bracketmatrix{{\xi_1}&\bar\xi_2\cr \xi_2&\bar \xi_1\cr}\bracketmatrix{{\xi_1}&\bar\xi_2\cr \xi_2&\bar \xi_1\cr}^*\right)^\frac12
\det\left(\bracketmatrix{{\xi_1}&\bar\xi_2\cr \xi_2&\bar \xi_1\cr}\bracketmatrix{{\xi_1}&\bar\xi_2\cr \xi_2&\bar \xi_1\cr}^*\right)^\frac12
\\
&= \left(\left(|\xi_1|^2-|\xi_2|^2\right)^2\right)^\frac12 \left(\left(|\xi_3|^2-|\xi_4|^2\right)^2\right)^\frac12 \\
&= \big||\xi_1|^2-|\xi_2|^2\big|\cdot\big||\xi_3|^2-|\xi_4|^2\big| \\
\\
\end{align*}
Identifying $W$ with $\mathbb{C}^4$, the two-point correlation of critical points on a Riemann surface is
\begin{equation}
\label{myk2}
K_2(z,w)=\intcf {e^{-\angleenclose{\Lambda^{-1}\vec\xi,\vec\xi}}\over\pi^6\det A\det\Lambda}
\big||\xi_1|^2-|\xi_2|^2\big|\cdot\big||\xi_3|^2-|\xi_4|^2\big|
\,d\vec\xi
\end{equation}

\section{The Scaling Limit}

Although the critical point equation $\nabla s(z)=0$ is not holomorphic, it is still smooth, so the results of \cite{bsz} about the zeroes of random smooth sections apply.  The main theorem (3.6) of \cite{bsz} actually says that as $N\to\infty$, the ``scaling limit'' of the correlation of zeroes is independent of choice of $M$, $L$, and~$h$. Specifically
\begin{equation}
{1\over N^{nk}}\knnk\left({z_1\over\sqrt N},\ldots,{z_n\over\sqrt N}\right)=\knkm(z_1,\ldots,z_n)+\bigo{1\over\sqrt N}
\end{equation}
where $\knkm$ depends only on $n$, $k$, and $m$.

If we write the 2-point correlation (\ref{myk2}) as $K_2^N$ to reflect the $N$ dependency in our case, this theorem says
\begin{equation}
{1\over N^2}K_2^N\left(\zovern,\wovern\right)=K^\infty_2(z,w)+\bigo{1\over\sqrt N}
\end{equation}
so proving Theorem \ref{main} only requires that we calculate the $N$ limit of the left hand side for $z$ close to $w$ for a particularly nice choice of $M$, $L$, and $h$.

Theorem 3.1 of \cite{bsz} says roughly that in the $N$ limit, the\break${1\over N}\Pi_N(\zovern,\wovern)$ entries in $K_2^N$ can be replaced by $\pih(\tilde z,\tilde w)$, the Szeg\H o kernel of the \it reduced Heisenberg Group\rm\ $\,\hred$ which we define below.  For a more thorough geometric discussion of $\hred$'s construction and properties, see \cite{bsz}\S1.3.2.

\begin{definition}\label{hred}
Take the trivial bundle $L:=\mathbb{C}\times\mathbb{C}$ over $\mathbb{C}$ with curvature $h(z):=e^{-|z|^2}$.  Then $h^{-1}(z)=e^{|z|^2}$ gives a metric on the dual bundle $L^*\to \mathbb{C}$.  Form the ``circle bundle'', $X$, of elements $v\in L^*$ such that $h^{-1}(v)=1$.  i.e. $X=\{(z,\zeta)\in\mathbb{C}\times\mathbb{C}\mid|\zeta|=e^{-{|z|^2\over2}}\}$.  This bundle $X\to \mathbb{C}$ is the \it reduced Heisenberg Group\rm, written $\hred$.  When necessary, since $X\iso\mathbb{C}\times S^1$, we will write elements as $(z,\theta)$.  Because $L\to \mathbb{C}$ here is the trivial bundle, we may use the frame $e_U=e_{\mathbb{C}}=1$, the constant function 1.
\end{definition}

The Szeg\H o kernel for $\hred$, $\pih$ is by definition the kernel of orthogonal projection from ${\cal L}^2(\hred)$ to the Hardy space for $\hred$, ${\cal H}_1^2$.  These spaces can be viewed as
\begin{equation}
{\cal L}^2(\hred)=\biggl\{\tilde f(z,\theta)=f(z)e^{i\theta}e^{-|z|^2\over2}\biggm|f\in{\cal C}^\infty, \intc f(z)e^{-|z|^2}\,dz\,\dzb<\infty\biggr\}
\end{equation}
and
\begin{equation}
{\cal H}^2_1=\biggl\{\tilde f(z,\theta)=f(z)e^{i\theta}e^{-|z|^2\over2}\biggm|f{\rm\ holomorphic}, \intc f(z)e^{-|z|^2}\,dz\,\dzb<\infty\biggr\}
\end{equation}
In fact,
$\phi_k(z):={1\over\sqrt{\pi k!}}z^ke^{i\theta}e^{-|z|^2\over2}$ are an orthonormal basis for ${\cal L}^2(\hred)$ since
\begin{equation}
{1\over\pi\sqrt{j!k!}}\intct z^j\zb^k\,e^{i(\theta-\phi)}e^{-|z|^2}\,\left({i\over2}\,dz\,\dzb\right)=\delta_{jk}
\end{equation}
so the kernel for the orthogonal projection ${\cal L}^2(\hred)\to{\cal H}^2_1$ is
\begin{equation}
\tilde\pih((z,\theta),(w,\phi)) {}=\sum_{k=0}^\infty\phi_k(z)\overline{\phi_k(w)} ={1\over\pi}e^{i(\theta-\phi)}e^{z\bar w-\frac12|z|^2-\frac12|w|^2}
\end{equation}
The extra factor $e^{\frac12|z|^2-\frac12|w|^2}$ appears because we have chosen a non-trivial metric $h$ for the trivial bundle.  As long as the connection is computed correctly, we could work in any frame, so the formula becomes
\begin{equation}
\pih(z,w)={1\over\pi}e^{z\bar w}
\end{equation}
based on the frame $e_{\mathbb{C}}=1$ (see Definition \ref{twopointkernel}).  Since (\ref{myk2}) doesn't change when $\pih(z,w)$ is multiplied by a non-zero scalar (hence likewise if $\Delta$ is multiplied by a non-zero scalar), we will use
\begin{equation}
\pih(z,w)=e^{z\bar w}
\end{equation}
for ease of calculation.

The Chern connection is defined by its action on a frame
\begin{equation}
\nablaHz e_{\mathbb{C}}(z)=\nablaHzh e_{\mathbb{C}}(z)=-\bar z\,dz\ot e_{\mathbb{C}}(z)
\end{equation}
Often, we will write $\pih((z,\theta),(w,\phi))$ and mean only the function coefficient $e^{z\bar w}$.

For our case, we form the Szeg\H o kernel for the $N$th power of an arbitrary positive line bundle over a Riemann surface 
$(L^N,h^N)\to M$ similarly, defining the circle bundle $$X_M:=\{s\in L^*\mid \angleenclose{s(z),s(z)}=1\}$$ and calling the Szeg\H o kernel
\begin{equation}
\pit:\tilde X\times\bar X\longrightarrow\mathbb{C}
\end{equation}
This Szeg\H o kernel is related to our earlier two-point function by
\begin{align*}
\tilde\Pi_N&\left(\left(\zovern,0\right),\left(\wovern,0\right)\right)=\cr
&\left(
\;
\Pi_N\left(\zovern,\wovern\right)
\;,\;\;h(z)^{N\over2}(e^*_U)^N(z)\ot h(w)^{N\over2}(e_U^*)^N(w)
\;
\right)
\end{align*}

With those definitions, we are in the position to state Theorem 3.1 of \cite{bsz} precisely in our case:

Choose $z_0\in M$, local coordinate map $z$, and a local frame $e_L$ over a neighborhood of $z_0$ so that 
\begin{equation}
\Theta_h(z_0)=(\del\delb\log h)(z_0)=dz\w\dzb|_{z_0}
\end{equation}
and
\begin{equation}
\ddz h(z_0)=\ddzdz h(z_0)=0
\end{equation}
then
\begin{align}
{1\over N}\tilde\Pi_N&\left(\left(z_0+{u\over\sqrt N},{\theta\over N}\right),\left(z_0+{v\over\sqrt N},{\phi\over N}\right)\right)=\cr
&\qquad\tilde\pih((u,\theta),(v,\phi))+\bigo{1\over\sqrt N}
\end{align}
where $\bigo{1\over\sqrt N}$ means a function whose ${\cal C}^k$-norm is $\bigo{1\over\sqrt N}$ in the standard sense for all $k$.

For
\begin{align}
  f_w(z)&:=\pit\bigg((z_0+z,0)\;,\;(z_0+w,0)\bigg) \\
  g_w(z)&:=f_w\left(\zovern\right)=\pit\bigg((z_0+\zovern,0)\;,\;(z_0+w,0)\bigg)\label{thatonethere}
\end{align}
\cite{bsz}'s theorem says
\begin{equation}
{1\over N}f_\wovern\left(\zovern\right)={1\over N}g_\wovern(z)=\pih((z,0),(w,0))+\bigo{1\over\sqrt N}
\end{equation}
Taking derivatives on both sides of (\ref{thatonethere})
\begin{equation}
g_w'(z)={1\over\sqrt N}f_w'\left(\zovern\right)
\end{equation}
so
\begin{align}
  {1\over N^\frac32}\left(\ddz{\pit}\right)&\left((z_0+\zovern,0),(z_0+\wovern,0)\right)\cr
  &={1\over N^\frac32}f_\wovern'\left(\zovern\right)={1\over N}g_\wovern'(z) \\
  &={1\over N}\ddz{\phantom{x}}\left(f_\wovern\left(\zovern\right)\right)=\ddz{\phantom x}\left({1\over N}f_\wovern\left(\zovern\right)\right) \\
  &=\ddz{\phantom x}\left(\pih((z,0),(w,0))+\bigo{1\over\sqrt N}\right)
\end{align}
Since we chose $h$ so that $\ddzop\log h=-\bar z+\bigo{|z|^2}$,
\begin{align*}
N\left(\ddzop\log h\right)\left(\zovern\right)&=N\left.(-\bar z+\bigo{|z|^2})\right|_\zovern\cr&=N\left(-{\bar z\over\sqrt N}+\bigo{|z|^2\over N}\right)\cr&=-\sqrt N\bar z+\bigo{|z|^2}
\end{align*}
so
\begin{align*}
  \left({1\over N^\frac32}\right.&\left.\nabz^N\Pi_N\right)\left(\zovern,\wovern\right)\cr
&={1\over N^\frac32}\left[d+N\del\log h\right]\Pi_N\left(\zovern,\wovern\right) \\
&={1\over N^\frac32}\ddz{\Pi_N}\left(\zovern,\wovern\right)+{1\over N^\frac32}N\del\log h\left(\zovern\right) \\
&={1\over N^\frac32}\ddz{\pit}\left((\zovern,0),(\wovern,0)\right)+{1\over N^\frac32}N\del\log h\left(\zovern\right) \\
&=\ddzop\left(\pih((z,0),(w,0))+\bigo{1\over\sqrt N}\right)+{1\over N^\frac32}(-\sqrt N\bar z+\bigo{|z|^2}) \\
&=\ddzop\left(\pih((z,0),(w,0))+\bigo{1\over\sqrt N}\right)-{1\over N}\bar z+{1\over N^\frac32}\bigo{|z|^2} \\
&=\nablaHz\pih((z,0),(w,0))+\bigo{1\over\sqrt N}
\end{align*}
so
\newcommand{\coro}{\bigo{1\over\sqrt N}}
\begin{align*}
  {1\over N^2}\A1p1{p'}&:={1\over N^2}\nabhz\nabyw\spppppp \cr
    &=\nabHzh\nabHwy\sppppp+\coro \cr
\end{align*}\begin{align*}
  {1\over N^\frac52}\B1p1{p'}1&:={1\over N^\frac52}\nabhz\nabyw\nabyw\spppppp \cr
  &=\nabHzh\nabHwy\nabHwy\sppppp+\coro  \cr
\end{align*}\begin{align*}
  {1\over N^\frac52}\B1p1{p'}2&:={1\over N^\frac52}\nabhz\spppppp\cr
&=\nabHzh\sppppp+\coro \cr
\end{align*}\begin{align*}
  {1\over N^3}&\C1p11{p'}1:={1\over N^3}\nabhz\nabhz\nabyw\nabyw\spppppp \cr
    &=\nabHzh\nabHzh\nabHwy\nabHwy\sppppp+\coro \cr
\end{align*}\begin{align*}
  {1\over N^3}\C1p11{p'}2&:={1\over N^3} \nabhz\nabhz\spppppp \cr
&=\nabHzh\nabHzh\sppppp+\coro \cr
\end{align*}\begin{align*}
  {1\over N^3}\C1p21{p'}1&:={1\over N^3} \nabyw\nabyw\spppppp \cr
    &=\nabHwy\nabHwy\sppppp+\coro\cr
\end{align*}\begin{align*}
  {1\over N^3}\C1p21{p'}2&:={1\over N^3} \spppppp \cr
&=\sppppp+\coro \cr
\end{align*}
and
\begin{equation}
{1\over N^3}\Lambda=\lamH+\coro
\end{equation}
where $\lamH$ is $\Lambda$ with all of the $\Pi_N$ terms replaced by $\pih$ terms and
\begin{align*}
  B&=\bracketmatrix{\bracketmatrix{\B1p1{p'}1}\bracketmatrix{\B1p1{p'}2}} \cr
  C&=
\bracketmatrix{
  \bracketmatrix{  \C1p11{p'}1 } &
  \bracketmatrix{  \C1p11{p'}2 } \cr
  \bracketmatrix{  \C1p21{p'}1 } &
  \bracketmatrix{  \C1p21{p'}2 } \cr
}
\cr
\end{align*}

So, finally,
\begin{align}
{1\over N^2}K_2^N\left(\zovern,\right.&\left.\wovern\right)\cr
&={N^{-2}\over\pi^6N^{16}\det\deltaH}\intcf e^{-N^{-3}\angleenclose{(\lamH)^{-1}\vec\xi,\vec\xi}}\|\det\vec\xi\|\,d\vec\xi
\end{align}
where $\deltaH$ is $\Delta$ with all of the $\Pi_N$ terms replaced by $\pih$ terms and $\|\det\vec\xi\|$ is shorthand for\break$|\det\xi^1||\det\xi^2|$.

Now perform the change of variables
\begin{equation}
\vec v:=\bracketmatrix{h_1\cr x_1\cr h_2\cr x_2}:=N^{-\frac32}\bracketmatrix{\xi_1\cr\xi_2\cr\xi_3\cr\xi_4}=N^{-\frac32}\vec\xi
\end{equation}
Now $\|\det\vec\xi\|=N^6\|\det\vec v\|$ and $d\vec\xi=N^{12}d\vec v$ so
\begin{equation}
  {1\over N^2}K_2^N\left(\zovern,\wovern\right)={1\over\pi^6\det\deltaH}\intcf e^{-\angleenclose{(\lamH)^{-1}\vec v,\vec v}}\|\det\vec v\|\,d\vec v
\end{equation}

\section{Additional Definitions and Notes}
\begin{definition}\label{crit} A \it critical point of $s\in H^0(M,L)$ with respect to $\nablauh$\rm\ is any $z\in M$ such that $\nablauh s(z)=0$.\end{definition}  For almost any $s\in H^0(M,L)$, the set
\begin{equation}
\mathop{\hbox{Crit}}\!\!{}^{\nablauh}(s):=\{z\in M\mid\nablauh s(z)=0\}
\end{equation}
is discrete, so the following definition makes sense.

\begin{definition}\label{Crit} The measure associated to $\mathop{\hbox{Crit}}^{\nablauh}(s)$ is
\begin{equation}
C^{\nablauh}_s:=\sum_{z\in\crit}\delta_z
\end{equation}
where $\delta_z$ is the point-mass at $z$.
\end{definition}

\begin{definition}\label{vol} The volume form associated to $h$, $dV_h$, is given by
\begin{equation}dV_h:={1\over m!}\left(-{i\over2}\del\delb\log h\right)^m\end{equation}\end{definition}
\begin{definition}\label{twopt} The \it two-point correlation\rm\ was not originally defined as the necessary integrand in the Kac-Rice formula.  It can be defined directly as
\begin{equation}
K_2(z,w)=\lim_{\e\to0}{ \E{\pound[\Crit\cap B_\e(z)]\cdot\pound[\Crit\cap B_\e(w)]\strut} \over \Vol{B_\e(z)\times B_\e(w)} }
\end{equation}
where
\begin{align}
B_\e(z):=&\hbox{the ball of radius $\e$ about $z$} \cr
\pound\!A:=&\hbox{the cardinality of $A$}
\end{align}
\end{definition}
$K_2(z,w)$ also comes from the distribution equation
\begin{equation}
\E{C^{\nablauh}_s\boxtimes C^{\nablauh}_s}=K_2(z,w)\;dV_h(z)\boxtimes dV_h(w)
\end{equation}
where $\boxtimes$ is the product on currents defined in \cite{numbervariance} by
\begin{equation}
S\boxtimes T=\pi_1^*S\w\pi_2^*T\in{\cal D}'^{p+q}(M\times M)
\end{equation}
for $S\in{\cal D}'^p(M)$ and $T\in{\cal D}'^q(M)$ where $\pi_1,\pi_2:M\times M\to M$ are the projections to the first and second factors, respectively.

\begin{definition}\label{bigoh} Throughout the calculation  $\bigoo n$, for $n\in[1,\infty)$, will mean a function $f$ such that 
\begin{equation*}
  \exists\delta,M>0\quad\st\quad\bigg(t\in(0,\delta)\implies|f(t)|\leq Mt^n\bigg)
\end{equation*}
In fact, every time it is used in this \doc, it is sufficient to think of $\bigoo n$ as a function real analytic at 0 whose first non-zero Taylor term when expanded there is a multiple of $t^n$.  i.e.
\begin{equation*}
f(t)=a_nt^n+a_{n+1}t^{n+1}+a_{n+2}t^{n+2}+\cdots
\end{equation*}
with $a_n\neq0$.  Technically, any time $\bigoo n$ is mentioned, it would be necessary to mention the radius of convergence, and often manipulation of a term involving $\bigoo n$ will result in a new term involving $\bigoo n$, where the radius of convergence has shrunk.  This calculation only requires that the radius of convergence stays positive.  As long as this is the case, the reader should not pay attention to this technicality.
\end{definition}

\section{Proof of the Main Result}
To prove Theorem \ref{main}, we want to find 
\begin{equation}
\kinf211(\zeta_1,\zeta_2)=\lim_{N\to\infty}{1\over N\npower}\kn21\left({\zeta_1\over\sqrt N},{\zeta_2\over\sqrt N}\right)
\end{equation}
for any $\zeta_1,\zeta_2\in M$ where $\dim_\mathbb{C} M=1$.  As above, Kac-Rice says we need only calculate
\begin{equation}
{1\over N^2}K_2^N\left(\zovern,\wovern\right)={1\over\pi^6\det\deltaH}\intcf e^{-\angleenclose{(\lamH)^{
-1}\vec v,\vec v}}\|\det\vec v\|\,d\vec v
\end{equation}
As in (\ref{splitdelta}), $\det\deltaH=(\det A^{\bf H})(\det\lamH)$  and since $K_{211}^\infty(z,w)$ depends only on the distance between $z$ and $w$, we can choose $z=0$ and $w=r>0$. So 
\begin{align*}
&\kinf211(z,w)=J(r):=\cr
&{1\over\pi^6\det(A( 0, r))\det(\Lambda( 0, r))}\intcf\!\!\left||\h1|^2-|x_1|^2\right|\!\cdot\!\left||\h2|^2-|x_2|^2\right|\etothelor dv
\end{align*}
where
\begin{equation}
v=\bracketmatrix{\h1\cr x_1\cr \h2\cr x_2}
\end{equation}

The absolute value bars simplification using Wick's formula as in\break\cite{bsz}, the fact that $\nabla s$ is not a holomorphic section bars using the Poincar\'e-Lelong formula as in \cite{pl}, and, unfortunately, we are unable to use the ingenious method used in the proof of Lemma 3.1 of \cite{cpsv2} where the authors were able to rewrite $J$ using Fourier transforms.  In \cite{cpsv2}, the authors noticed that they could replace each $\displaystyle\big||h_j|^2-|x_j|^2\big|$ by 
\begin{equation}
\lim_{\e_j,\e'_{\!j}\to0}{1\over2\pi}\intr\intr|p|e^{-\e_j|\xi|^2-\e'_{\!j}|p|^2}\cdot e^{i\xi(p-|h_j|^2+|x_j|^2)}\,d\xi\,dp
\end{equation}
because it can be simplified to 
\begin{align*}
&\limep{1\over2\pi}\intr|p|e^{-\e'_{\!j}|p|^2}\left(\intr \lime e^{i\xi(p-|h_j|^2+|x_j|^2)}e^{\e_j|\xi|^2}\,d\xi\right)\,dp \cr
=&\limep{1\over2\pi}\intr|p|e^{-\e'_{\!j}|p|^2}\left(\intr \lime e^{i\xi(p-|h_j|^2+|x_j|^2)}\,d\xi\right)\,dp \cr
=&\limep{1\over2\pi}\intr|p|e^{-\e'_{\!j}|p|^2}\big(2\pi\delta_0(p-|h_j|^2+|x_j|^2)\big)\,dp \cr
=&\limep\intr|p|e^{-\e'_{\!j}|p|^2}\delta_{|h_j|^2-|x_j|^2}(p)\,dp \cr
=&\limep\big||h_j|^2-|x_j|^2\big|\,e^{-\e'_{\!j}\left||h_j|^2-|x_j|^2\right|^2} \cr
=&\big||h_j|^2-|x_j|^2\big| \cr
\end{align*}
With that substitution and some work,
\begin{equation}
J=\limes\je={1\over4\pi^2}\intrf{|p_1p_2|e^{i(\xi_1p_1)}e^{i(\xi_2p_2)}\over\detildi}\,d\xi_1\,d\xi_2\,dp_1\,dp_2
\end{equation}
with
\begin{equation}
D=\bracketmatrix{-\xi_1\cr&\xi_1\cr&&-\xi_2\cr&&&\xi_2}
\end{equation}

In the case of \cite{cpsv2}, the authors took advantage of the fact that $\detildi$ was the product of many linear factors and the integral could be done using residues.  In our case $\detildi$ is an extremely complicated rational function of $r$ and $e^{r^2}$.  So in this \doc\ we will carefully expand $J$ as a function of $r$.

Calculate $\Lambda$ as in \cite{bsz} and \cite{cpsv1}:

\begin{align*}
  \Lambda&:=C-B^*A^{-1}B\cr
   A&:=\bracketmatrix{\A jp{j'}{p'}}=\bracketmatrix{\E {x_j^p\bar x_{j'}^{p'}}}=\bracketmatrix{ \A1111 & \A1112 \cr\cr \A1211 & \A1212 } \cr
   B&:=\bracketmatrix{\B jp{j'}{p'}{q'}}=\bracketmatrix{\E{x_j^p\bar\xi_{j'q'}^{p'}}}=\bracketmatrix{ \B11111&\B11112&\B11121&\B11122\cr\cr \B12111&\B12112&\B12121&\B12122\cr } \cr
   C&:=\bracketmatrix{\C jpq{j'}{p'}{q'}}=\bracketmatrix{\E{\xi_{jq}^p\bar\xi_{j'q'}^{p'}}}=\bracketmatrix{ \C111111&\C111112&\C111121&\C111122\cr\cr \C112111&\C112112&\C112121&\C112122\cr\cr \C121111&\C121112&\C121121&\C121122\cr\cr \C122111&\C122112&\C122121&\C122122\cr } \cr
\end{align*}
where $p,p'\in\{1,2\}$\quad$q,q'\in\{1,2\}$ and the $p,q$ index the rows, and the $p',q'$ index the columns.

\begin{align*}
  \A1p1{p'}&:=\nabhz\nabyw\spp \cr
  \B1p1{p'}1&:=\nabhz\nabyw\nabyw\spp \cr
  \B1p1{p'}2&:=\nabhz\spp \cr
  \C1p11{p'}1&:=\nabhz\nabhz\nabyw\nabyw\spp \cr
  \C1p11{p'}2&:= \nabhz\nabhz\spp\cr
  \C1p21{p'}1&:= \nabyw\nabyw\spp\cr
  \C1p21{p'}2&:= \spp\cr
\end{align*}
Notice $p,q$ index rows and $p',q'$ index columns.

For \it any\rm\ function $\pinn$ holomorphic in $z$ and antiholomorphic in $w$ where
\begin{align*}
  \nabz e(z)&=\nabhz e(z)=g(z)\,dz\tay\ez \cr
  \nabz\ebz&=\nabyz\ebz=\overline{g(z)}\,\dzb\tay\ebz \cr
  \nabyz\ez&=0 \cr
  \nabhz\ebz&=0 \cr
\end{align*}
we have
\begin{align*}
  \nabhz(&\orig)\cr
  &=\left[\ddz\Pi\tay\dz\right]\tay\ezebw+\pinn\tay\gdzez\tay\ebw \cr
  &=\left(\ddz\Pi+g(z)\pinn\right)\tay\dz\tay\ezebw \cr
\end{align*}\begin{align*}
  \nabhz\nabhz(&\orig)\cr
  &=\left[\left(\ddzdz\Pi+\ddz g\pinn+g(z)\ddz\Pi\right)\tay\dz\right]\tay\dz\tay\ezebw \cr
  &\phantom={}+\left(\ddz\Pi+g(z)\pinn\right)\tay\dz\tay\gdzez\tay\ebw \cr
  &=\left(\ddzdz\Pi+\ddz g\pinn+g(z)\ddz\Pi+g(z)\ddz\Pi+g(z)^2\pinn\right) \cr
  &\qquad\qquad\tay\dz\tay\dz\tay\ezebw \cr
  &=\left(\ddzdz\Pi+\ddz g\pinn+2g(z)\ddz\Pi+g(z)^2\pinn\right)\cr
  &\qquad\qquad\tay\dz\tay\dz\tay\ezebw \cr
\end{align*}\begin{align*}
  \nabyw(&\orig)\cr
  &=\left[\ddwb\Pi\tay\dwb\right]\tay\ezebw+\pinn\tay\ez\tay\gdwbebw \cr
  &=\left(\ddwb\Pi+\overline{g(w)}\pinn\right)\tay\dwb\tay\ezebw \cr
\end{align*}\begin{align*}
  \nabyw&\nabyw(\orig)\cr
  &=\left[\left(\ddwbdwb\Pi+\ddwb{\bar g}\pinn+\overline{g(w)}\ddwb\Pi\right)\tay\dwb\right]\tay\dwb\tay\ezebw \cr
  &=\left(\ddwb\Pi+\overline{g(w)}\pinn\right)\tay\dwb\tay\ez\tay\gdwbebw \cr
  &=\left(\ddwbdwb\Pi+\ddwb{\bar g}\pinn+2\overline{g(w)}\ddwb\Pi+\overline{g(w)}^2\pinn\right)\cr
  &\qquad\qquad\tay\dwb\tay\dwb\tay\ezebw \cr \noalign{\break}
\end{align*}\begin{align*}
  \nabhz&\nabyw(\orig)\cr
  &=\left[\left(\ddzdwb\Pi+\overline{g(w)}\ddz\Pi\right)\tay\dz\right]\tay\dwb\tay\ezebw \cr
  &\phantom={}+\left(\ddwb\Pi+\overline{g(w)}\pinn\right)\tay\dwb\tay\gdzez\tay\ebw \cr
  &=\left(\ddzdwb\Pi+g(z)\ddwb\Pi+\overline{g(w)}\ddz\Pi+g(z)\overline{g(w)}\pinn\right)\cr
  &\qquad\qquad\tay\dz\tay\dwb\tay\ezebw \cr
\end{align*}\begin{align*}
  \nabhz&\nabyw\nabyw(\orig)\cr
  &=\left[\left(\ddzdwbdwb\Pi+\ddwb{\bar g}\ddz\Pi+2\overline{g(w)}\ddzdwb\Pi+\overline{g(w)}^2\ddz\Pi\right)\tay\dz\right]\cr
  &\qquad\qquad\tay\dwb\tay\dwb\tay\ezebw \cr
  &\phantom={}+\left(\ddwbdwb\Pi+\ddwb{\bar g}\pinn+2\overline{g(w)}\ddwb\Pi+\overline{g(w)}^2\pinn\right)\cr
  &\qquad\qquad\tay\dwb\tay\dwb\tay\gdzez\ebw \cr
  &=\left(
\begin{aligned}
  &\ddzdwbdwb\Pi+g(z)\ddwbdwb\Pi+2\overline{g(w)}\ddzdwb\Pi+\left(\ddwb{\bar g}+\overline{g(w)}^2\right)\ddz\Pi \cr
  +&2g(z)\overline{g(w)}\ddwb\Pi+g(z)\left(\ddwb{\bar g}+\overline{g(w)}^2\right)\pinn
\end{aligned}
\right) \cr
&\qquad\qquad\tay\dz\tay\dwb\tay\dwb\tay\ezebw \cr
\end{align*}

In this particular case where $g(z)=-\zb$ because $h(z):=e^{-z\zb}$, $\pinn=e^{z\wb}$, $\zz1:=0$, and $\zz2:=r$, we~have
\begin{align}
  \A1p1{p'}&={e^{z\bar{ w } }} \left( 1+z\bar{ w } -\bar{ z } z-w\bar{ w } +\bar{ z } w \right) \cr
  \B1p1{p'}1&={e^{z\bar{ w } }} \left( z-w \right)  \left( z \bar w -\bar{ z } z+2+\bar{ z } w-w\bar{ w }  \right)\cr
  \B1p1{p'}2&={e^{z\bar{ w } }} \left( \bar{ w } -\bar{ z }  \right)\cr
  \C1p11{p'}1&={e^{z\bar{ w } }} \left(
\begin{aligned}
  &2-4\,\bar{ z } z-4\,w\bar{ w } +4\,\bar{ z } w-2\,\bar{ w } \bar{ z } {z}^{2}-2\, \bar w ^{2}zw-2\,\bar{ w } \bar{ z } {w}^{2} \cr
  -&2\, \bar z ^{2}zw +4\,z\bar{ w } + \bar w ^{2}{z}^{2}+ \bar w ^{2}{w}^{2}+ \bar z ^{2}{z}^{2}+ \bar z ^{2}{w}^{2}+4\,\bar{ w } z\bar{ z } w
\end{aligned}
\right)\cr
  \C1p11{p'}2&={e^{z\bar{ w } }} \left( \bar{ w } -\bar{ z }  \right) ^{2}\cr
  \C1p21{p'}1&={e^{z\bar{ w } }} \left( z-w \right) ^{2}\cr
  \C1p21{p'}2&={e^{z\bar{ w } }}\cr
\end{align}
so
\begin{align}
A&=\bracketmatrix{ 1&1-{r}^{2}\cr \noalign{\medskip}1-{r}^{2}&{e^{{r}^{2}}}}\cr
B&=\bracketmatrix{0&0&-r\, \left( 2-{r}^{2} \right) &r\cr \noalign{\medskip}r\, \left( 2-{r}^{2} \right) &-r&0&0}\cr
C&=\bracketmatrix{2&0&-4\,{r}^{2}+{r}^{4}+2&{r}^{2}\cr \noalign{\medskip}0&1&{r}^{2}&1\cr \noalign{\medskip}-4\,{r}^{2}+{r}^{4}+2&{r}^{2}&2\,{e^{{r}^{2}}}&0\cr \noalign{\medskip}{r}^{2}&1&0&{e^{{r}^{2}}}}
\end{align}
so $$\Lambda(r)={1\over-{e^{{r}^{2}}}+1-2\,{r}^{2}+{r}^{4}}\cdot\genm1{}{}$$ where
$$
\begin{aligned}
\genmr111&=-2\,{e^{{r}^{2}}}+2-2\,{r}^{4}+{r}^{6} \cr
\genmr121&={r}^{2} \left( -2+{r}^{2} \right) \cr
\genmr131&= 4\,{r}^{2}{e^{{r}^{2}}}-4\,{r}^{2}+3\,{r}^{4}-{r}^{6}-{r}^{4}{e^{{r}^{2}}}-2\,{e^{{r}^{2}}}+2\cr
\genmr141&= -{r}^{2} \left( {e^{{r}^{2}}}+1-{r}^{2} \right) \cr
\end{aligned}
$$
$$
\begin{aligned}
\genmr112&={r}^{2} \left( -2+{r}^{2} \right) \cr
\genmr122&=-{e^{{r}^{2}}}+1-{r}^{2}+{r}^{4}\cr
\genmr132&=-{r}^{2} \left( {e^{{r}^{2}}}+1-{r}^{2} \right) \cr
\genmr142&=-{e^{{r}^{2}}}+1-{r}^{2}\cr
\end{aligned}
$$
$$
\begin{aligned}
\genmr113&=4\,{r}^{2}{e^{{r}^{2}}}-4\,{r}^{2}+3\,{r}^{4}-{r}^{6}-{r}^{4}{e^{{r}^{2}}}-2\,{e^{{r}^{2}}}+2\cr
\genmr123&=-{r}^{2} \left( {e^{{r}^{2}}}+1-{r}^{2} \right) \cr
\genmr133&=-{e^{{r}^{2}}} \left( 2\,{e^{{r}^{2}}}-2+2\,{r}^{4}-{r}^{6} \right) \cr
\genmr143&={r}^{2} \left( -2+{r}^{2} \right) {e^{{r}^{2}}}\cr
\end{aligned}
$$
$$
\begin{aligned}
\genmr114&=-{r}^{2} \left( {e^{{r}^{2}}}+1-{r}^{2} \right) \cr
\genmr124&=-{e^{{r}^{2}}}+1-{r}^{2}\cr
\genmr134&={r}^{2} \left( -2+{r}^{2} \right) {e^{{r}^{2}}}\cr
\genmr144&=-{e^{{r}^{2}}} \left( {e^{{r}^{2}}}-1+{r}^{2}-{r}^{4} \right)
\end{aligned}
$$
which is actually a function of $t:=r^2>0$.
\goodbreak
\noindent i.e. $$\Lambda(t)={1\over-{e^{t}}+1-2\,t+{t}^{2} }\cdot\genm2{}{}$$
where
$$
  \begin{aligned}
  \genm211&=-2\,{e^{t}}+2-2\,{t}^{2}+{t}^{3} \cr
  \genm221&=t\, \left( -2+t \right)  \cr
  \genm231&=4\,t\,{e^{t}}-4\,t+3\,{t}^{2}-{t}^{3}-{t}^{2}{e^{t}}-2\,{e^{t}}+2 \cr
  \genm241&=t\, \left( -{e^{t}}-1+t \right)  \cr
  \end{aligned}
$$
$$
  \begin{aligned}
  \genm212&=t\, \left( -2+t \right)  \cr
  \genm222&=-{e^{t}}+1-t+{t}^{2} \cr
  \genm232&=t\, \left( -{e^{t}}-1+t \right)  \cr
  \genm242&=-{e^{t}}+1-t \cr
  \end{aligned}
$$
$$
  \begin{aligned}
  \genm213&=4\,t\,{e^{t}}-4\,t+3\,{t}^{2}-{t}^{3}-{t}^{2}{e^{t}}-2\,{e^{t}}+2 \cr
  \genm223&=t\, \left( -{e^{t}}-1+t \right)  \cr
  \genm233&={e^{t}} \left( -2\,{e^{t}}+2-2\,{t}^{2}+{t}^{3} \right)  \cr
  \genm243&=t\, \left( -2+t \right) {e^{t}} \cr
  \end{aligned}
$$
$$
\begin{aligned}
\genm214&=t\, \left( -{e^{t}}-1+t \right) \cr
\genm224&=-{e^{t}}+1-t\cr
\genm234&=t\, \left( -2+t \right) {e^{t}}\cr
\genm244&={e^{t}} \left( -{e^{t}}+1-t+{t}^{2} \right)
\end{aligned}
$$
so
\begin{align*}
&\det\Lambda(t)=\cr
&{\frac { \left( {e^{t}} \right) ^{2}{t}^{4}\!-\!{e^{t}}{t}^{4}\!-\!4{t}^{3}{e^{t}}\!-\!4{t}^{3} \left( {e^{t}} \right) ^{2}\!+\!\!12{t}^{2} \left( {e^{t}} \right) ^{2}\!\!-\!12{e^{t}}{t}^{2}\!-\!12{e^{t}}\!+\!12 \left( {e^{t}} \right) ^{2}\!-\!4 \left( {e^{t}} \right) ^{3}\!\!+\!4}{-{e^{t}}+1-2t+{t}^{2}}}
\end{align*}
\begin{equation}= {\frac {1}{6480}}{t}^{8}+{\frac {1}{3888}}{t}^{9}+{\frac {869}{4082400}}{t}^{10}+{\frac {37}{326592}}{t}^{11}+{\frac {1213}{29393280}}{t}^{12}+\bigoo{13} 
\label{detexpan}
\end{equation}
\begin{align}
\label{detaexpan}
\det A(t)&=e^t-t^2+2t-1\cr
&=3t-\frac12t^2+\frac16t^3+{1\over24}t^4+\bigoo5
\end{align}

Now
\begin{align*}&\Lambda^{-1}(t)=\cr
&{1\over\!-\!{t}^{4}{e^{t}}\!+\!{t}^{4}{e^{2t}}\!-\!4{t}^{3}{e^{t}}\!-\!4{t}^{3}{e^{2t}}\!-\!12{t}^{2}{e^{t}}\!+\!12{t}^{2}{e^{2t}}\!-\!4{e^{3t}}\!\!-\!12{e^{t}}\!\!+\!12{e^{2t}}\!+\!4}\!\cdot\!\genm3{}{}\end{align*}
where
\begin{align*}
\genm311&={\scriptstyle {t}^{3} \left( {e^{t}} \right) ^{2}-{t}^{2} \left( {e^{t }} \right) ^{2}-{e^{t}}{t}^{2}+4\,t\, \left( {e^{t}} \right) ^{2}-4\,t\,{e^{t}}-2 \, \left( {e^{t}} \right) ^{3}+4\, \left( {e^{t}} \right) ^{2}-2\,{e^{t}} } \cr
\genm321&=-{t}^{3}{e^{t}}-{t}^{3} \left( {e^{t}} \right) ^{2}+2\,{t}^{2} \left( {e^{t}} \right) ^{2}-2\,{e^{t}}{t}^{2}\cr
\genm331&={t}^{3}{e^{t}}+{ t}^{2} \left( {e^{t}} \right) ^{2}+{e^{t}}{t}^{2}+4\,t\,{e^{t}}-4\,t\, \left( {e^{ t}} \right) ^{2}+2\, \left( {e^{t}} \right) ^{2}-4\,{e^{t}}+2\cr
\genm341&=- 2\,{t}^{3}{e^{t}}+2\,t\, \left( {e^{t}} \right) ^{2}-4\,t\,{e^{t}}+2\,t \cr
\genm312&=-{t}^{3}{e^{t}}-{t}^{3} \left( {e^{t}} \right) ^{2}+2\,{ t}^{2} \left( {e^{t}} \right) ^{2}-2\,{e^{t}}{t}^{2}\cr
\genm322&=\scriptstyle \left( {e ^{t}} \right) ^{2}{t}^{4}-{e^{t}}{t}^{4}-4\,{t}^{3} \left( {e^{t}} \right) ^{2}-2 \,{e^{t}}{t}^{2}+10\,{t}^{2} \left( {e^{t}} \right) ^{2}-4\,t\, \left( {e^{t}} \right) ^{2}+4\,t\,{e^{t}}-4\, \left( {e^{t}} \right) ^{3}+8\, \left( {e^{t}} \right) ^{2}-4\,{e^{t}}\cr
\genm332&=-2\,{t}^{3}{e^{t}}+2\,t\, \left( {e^{t }} \right) ^{2}-4\,t\,{e^{t}}+2\,t\cr
\genm342&=4\,{t}^{3}{e^{t}}+2\,{t}^{2} -10\,{e^{t}}{t}^{2}+4\,t\,{e^{t}}-4\,t+4\, \left( {e^{t}} \right) ^{2}-8\,{e^{t}}+ 4 \cr
\genm313&={t}^{3}{e^{t}}+{t}^{2} \left( {e^{t}} \right) ^{2}+{e^{t }}{t}^{2}+4\,t\,{e^{t}}-4\,t\, \left( {e^{t}} \right) ^{2}+2\, \left( {e^{t}} \right) ^{2}-4\,{e^{t}}+2\cr
\genm323&=-2\,{t}^{3}{e^{t}}+2\,t\, \left( {e^ {t}} \right) ^{2}-4\,t\,{e^{t}}+2\,t\cr
\genm333&={t}^{3}{e^{t}}-{e^{t}}{t}^ {2}-{t}^{2}+4\,t\,{e^{t}}-4\,t-2\, \left( {e^{t}} \right) ^{2}+4\,{e^{t}}-2 \cr
\genm343&=-{t}^{3}-{t}^{3}{e^{t}}+2\,{e^{t}}{t}^{2}-2\,{t}^{2} \cr
\genm314&=-2\,{t}^{3}{e^{t}}+2\,t\, \left( {e^{t}} \right) ^{2}-4 \,t\,{e^{t}}+2\,t\cr
\genm324&=4\,{t}^{3}{e^{t}}+2\,{t}^{2}-10\,{e^{t}}{t}^{ 2}+4\,t\,{e^{t}}-4\,t+4\, \left( {e^{t}} \right) ^{2}-8\,{e^{t}}+4 \cr
\genm334&=-{t}^{3}-{t}^{3}{e^{t}}+2\,{e^{t}}{t}^{2}-2\,{t}^{2} \cr
\genm344&={e^{t}}{t}^{4}-{t}^{4}-4\,{t}^{3}{e^{t}}-2\,{t}^{2}+10\,{e^{t} }{t}^{2}-4\,t\,{e^{t}}+4\,t-4\, \left( {e^{t}} \right) ^{2}+8\,{e^{t}}-4
\end{align*}
so
\goodbreak
\begin{equation}
\Lambda^{-1}=t^{-5}\cdot Y(t)
\end{equation}
where
$$
\begin{aligned}
\yy11&=30  {t}^{2}+9  {t}^{3}+\bigoo4  \cr
\yy21&=360  t-{\frac {18}{7}}{t}^{3}+\bigoo4 \cr
\yy31&=-30  {t}^{2}+9  {t}^{3}+\bigoo4 \cr
\yy41&=-360  t+180  {t}^{2}-{\frac {318}{7}}{t}^{3}+\bigoo4 \cr
\end{aligned}
$$
$$
\begin{aligned}
\yy12&=360  t-{\frac {18}{7}}{t}^{3}+\bigoo4  \cr
\yy22&=4320-1080  t+{\frac {960}{7}}{t}^{2}-{\frac {72}{7}}{t}^{3}+\bigoo4 \cr
\yy32&=-360  t+180  {t}^{2}-{\frac {318}{7}}{t}^{3}+\bigoo4 \cr
\yy42&=-4320+3240  t-{\frac {8520}{7}}{t}^{2}+{\frac {2148}{7}}{t}^{3}+\bigoo4 \cr
\end{aligned}
$$
$$
\begin{aligned}
\yy13&=-30  {t}^{2}+9  {t}^{3}+\bigoo4  \cr
\yy23&=-360  t+180  {t}^{2}-{\frac {318}{7}}{t}^{3}+\bigoo4 \cr
\yy33&=30  {t}^{2}-21  {t}^{3}+\bigoo4 \cr
\yy43&=360  t-360  {t}^{2}+{\frac {1242}{7}}{t}^{3}+\bigoo4 \cr
\end{aligned}
$$
$$
\begin{aligned}
\yy14&=-360  t+180  {t}^{2}-{\frac {318}{7}}{t}^{3}+\bigoo4 \cr
\yy24&=-4320+3240  t-{\frac {8520}{7}}{t}^{2}+{\frac {2148}{7}}{t}^{3}+\bigoo4\cr
\yy34&=360  t-360  {t}^{2}+{\frac {1242}{7}}{t}^{3}+\bigoo4\cr
\yy44&=4320-5400  t+{\frac {23640}{7}}{t}^{2}-{\frac {9852}{7}}{t}^{3}+\bigoo4 
\end{aligned}
$$
Note
\begin{equation}
\lim_{t\to0^+}Y(t)=\bracketmatrix{0&0&0&0\cr 0&4320&0&-4320\cr 0&0&0&0\cr 0&-4320&0&4320}
\end{equation}
Now the original integral can be estimated by estimating the diagonalization of $Y(t)$.  Though the proof does not depend on knowing the origin of the $U(t)$ and $D(t)$ used to approximate diagonalizing $Y(t)$, their construction is given in the appendix of the arxiv version of this paper \cite{arxiv_version} which also includes associated maple code and output.
\begin{align*}
  &J(t):=\cr
&\phantom=\lefty\intcf\left||\h1|^2-|x_1|^2\right|\cdot\left||\h2|^2-|x_2|^2\right|\etothelt\,dv \cr
  &=\lefty\intcf\left||\h1|^2-|x_1|^2\right|\cdot\left||\h2|^2-|x_2|^2\right|e^{-t^{-5}\langle Y(t)v,v\rangle}\,dv \cr
  &=\lefty\intcf\!\!\!\!\left||\h1|^2-|x_1|^2\right|\!\!\cdot\!\!\left||\h2|^2-|x_2|^2\right|e^{-\big\langle Y(t)(t^{-\frac52}v),(t^{-\frac52}v)\big\rangle} dv \cr
\end{align*}
Making the substitution $w=t^{-\frac52}v$ is actually saying
\begin{align*}
&\phantom\implies\bracketmatrix{w_1\cr w_2\cr w_3\cr w_4}=t^{-\frac52}\bracketmatrix{h_1\cr x_1\cr h_2\cr x_2}\cr
&\implies
\bracketmatrix{\dform w1\cr\dform w2\cr\dform w3\cr\dform w4}=t^{-\frac52}\bracketmatrix{\dform h1\cr\dform x1\cr\dform h2\cr\dform x2}
\cr
&\implies
dw=\frac i2\dform w1\dwb_1\dots\frac i2\dform w4\dwb_4=t^{-20}\frac i2\dform h1 d\bar h_1\dots\frac i2\dform x4 d\bar x_4=t^{-20}dv
\end{align*}
making
\begin{align*}
  &J(t)=\cr
  &\lefty\intcf\left||\h1|^2-|x_1|^2\right|\cdot\left||\h2|^2-|x_2|^2\right|e^{-\angleenclose{ Y(t)(t^{-\frac52}v),(t^{-\frac52}v)}}dv \cr
  &=\!\lefty\!\intcf\!\!\!t^5\!\left||w_1|^2-|w_2|^2\right|\cdot t^5\left||w_3|^2-|w_4|^2\right|\!e^{-\angleenclose{Y(t)w,w}}t^{20}dw \cr
  &=\!\sidey\intcf \left||w_1|^2-|w_2|^2\right|\cdot \left||w_3|^2-|w_4|^2\right|e^{-\angleenclose{Y(t)w,w}}\,dw \cr
\end{align*}
Now
\begin{align*}
Y(t)&=U(t)^*D(t)U(t)\cr
&=\bigg(\tilde U(t)^*+\oaf3^*\bigg)\bigg(\tilde D(t)+\daf{12}\bigg)\bigg(\tilde U(t)+\oaf3\bigg) \cr
&=\tilde U(t)^*\tilde D(t)\tilde U(t)+\oaf3
\end{align*}
where
\begin{align*}
D(t):&=\bracketmatrix{\lll1&0&0&0\cr0&\lll2&0&0\cr0&0&\lll3&0\cr0&0&0&\lll4}\cr
&=
\underbrace{\bracketmatrix{\lt1&0&0&0\cr0&\lt2&0&0\cr0&0&\lt3&0\cr0&0&0&\lt4}}_{\displaystyle\tilde D(t)}+\underbrace{\omat{12}}_{\displaystyle\daf{12}}
\end{align*}
where
\begin{align*}
\lt1:=&8640-6480\,t+{\frac {25020}{7}}\,{t}^{2}-{\frac {10050}{7}}\,{t}^{3}+{\frac {66380}{147}}\,{t}^{4}-{\frac {261767}{2352}}\,{t}^{5}\cr
&+{\frac {48960935}{2173248}}\,{t}^{6}-{\frac {29628553}{8149680}}\,{t}^{7}
+{\frac {208429618963}{427173626880}}\,{t}^{8}\cr&-{\frac {560822276587}{8543472537600}}\,{t}^{9}+{\frac {46335059891}{6133775155200}}\,{t}^{10}\cr&+{\frac {518190034231}{1794129232896000}}\,{t}^{11}  \cr
\lt2:=&6\,{t}^{3}-3\,{t}^{4}+{\frac {111}{80}}\,{t}^{5}-{\frac {161}{960}}\,{t}^{6}-{\frac {20561}{1209600}}\,{t}^{7}+{\frac {561019}{21772800}}\,{t}^{8}\cr&+{\frac {3916753}{15676416000}}\,{t}^{9}
-{\frac {827998967}{282175488000}}\,{t}^{10}\cr&+{\frac {5185091420987}{15643809054720000}}\,{t}^{11}  \cr
\lt3:=&\frac13\,{t}^{4}-{1\over12}\,{t}^{5}+{\frac {1}{72}}\,{t}^{6}+{1\over32}\,{t}^{7}+{\frac {6223}{207360}}\,{t}^{8}+{\frac {256685}{8957952}}\,{t}^{9}\cr&+{\frac {588107563}{22574039040}}\,{t}^{10}+{\frac {6399891227}{325066162176}}\,{t}^{11}  \cr
\lt4:=&{3\over8}\,{t}^{5}-{1\over16}\,{t}^{6}-{\frac {65}{768}}\,{t}^{7}-{\frac {101}{3072}}\,{t}^{8}-{\frac {877}{40960}}\,{t}^{9}\cr&-{\frac {37303}{1474560}}\,{t}^{10}-{\frac {2563021}{123863040}}\,{t}^{11}  \cr
\end{align*}
and
\begin{equation}
U(t)=\tilde U(t)+\underbrace{\allomat3}_{\displaystyle\oaf3}
\end{equation}
where
$$
\begin{aligned}
\tilde u_{11}=& -{\sqrt2\over24}t-{\sqrt2\over48}t^2  \cr
\tilde u_{21}=& -{\frac {1}{2}}-{\frac {3}{16}}t-{\frac {65}{2304}}{t}^{2}    \cr
\tilde u_{31}=& {\sqrt2\over2}+0t+{43\sqrt2\over576}t^2  \cr
\tilde u_{41}=& {\frac {1}{2}}-{\frac {3}{16}}t-{\frac {193}{768}}{t}^{2}    \cr
\end{aligned}
$$
$$
\begin{aligned}
\tilde u_{12}=&  -{\sqrt2\over2}-{\sqrt2\over8}t+{5\sqrt2\over288}t^2  \cr
\tilde u_{22}=& {\frac {1}{2}}-{\frac {3}{16}}t+{\frac {5}{256}}{t}^{2}    \cr
\tilde u_{32}=&{\sqrt2\over12}+0t+{61\sqrt2\over576}t^2  \cr
\tilde u_{42}=& {\frac {1}{2}}-{\frac {1}{16}}t-{\frac {53}{768}}{t}^{2}    \cr
\end{aligned}
$$
$$
\begin{aligned}
\tilde u_{13}=&{\sqrt2\over24}t+0t^2  \cr
\tilde u_{23}=&  -{\frac {1}{2}}+{\frac {1}{16}}t+{\frac {167}{2304}}{t}^{2}    \cr
\tilde u_{33}=& -{\sqrt2\over2}+0t+{7\sqrt2\over64}t^2  \cr
\tilde u_{43}=&  {\frac {1}{2}}+{\frac {1}{16}}t+{\frac {215}{768}}{t}^{2}    \cr
\end{aligned}
$$
$$
\begin{aligned}
\tilde u_{14}=& {\sqrt2\over2}-{\sqrt2\over8}t-{5\sqrt2\over288}t^2 \cr
\tilde u_{24}=&  {\frac {1}{2}}+{\frac {1}{16}}t-{\frac {41}{768}}{t}^{2} \cr
\tilde u_{34}=& {\sqrt2\over6}t+{109\sqrt2\over576}t^2 \cr
\tilde u_{44}=&  {\frac {1}{2}}+{\frac {3}{16}}t-{\frac {29}{768}}{t}^{2}  \cr
\end{aligned}
$$
and $\oaf3$ taken so that $U(t)$ is invertible.

The following lemmas say that $U(t)$ is approximately orthogonal.

\begin{lemma}\label{uustarI} $U(t)U(t)^*=I+\oaf3$\end{lemma}
\begin{proof}
\begin{align*}
&U(t)U(t)^*\cr
&=\bigg(\tilde U(t)+\oaf3\bigg)\bigg(\tilde U(t)+\oaf3\bigg)^* \cr
&=\bigg(\ut+\oaf3\bigg)\bigg(\ut^*+\oaf3^*\bigg) \cr
&=\ut\ut^*+\oaf3\ut^*+\ut\oaf3+\oaf3\oaf3 \cr
&=\ut\ut^*+\oaf3+\oaf3+\oaf6 \cr
&=\ut\ut^*+\oaf3 \cr
&=\genm\mn{}{}+\oaf3 \cr
\end{align*}
where

\[
\begin{aligned}
\genm\mn11&=1+{\frac {1}{288}}\,{t}^{3}+{\frac {43}{20736}}\,{t}^{4} \cr
\genm\mn21&={\frac {221}{27648}}\,\sqrt {2}{t}^{3}+{\frac {205}{110592}}\,\sqrt {2}{t}^{4} \cr
\genm\mn31&=-{\frac {85}{1152}}\,{t}^{3}-{\frac {83}{13824}}\,{t}^{4} \cr
\genm\mn41&={\frac {323}{9216}}\,\sqrt {2}{t}^{3}+{\frac {173}{36864}}\,\sqrt {2}{t}^{4}
\end{aligned}
\]\[
\begin{aligned}
\genm\mn12&={\frac {221}{27648}}\,\sqrt {2}{t}^{3}+{\frac {205}{110592}}\,\sqrt {2}{t}^{4} \cr
\genm\mn22&=1+{\frac {13}{2304}}\,{t}^{3}+{\frac {12317}{1327104}}\,{t}^{4} \cr
\genm\mn32&=-{\frac {23}{1024}}\,\sqrt {2}{t}^{3}-{\frac {367}{165888}}\,\sqrt {2}{t}^{4} \cr
\genm\mn42&={\frac {85}{1152}}\,{t}^{3}+{\frac {517}{18432}}\,{t}^{4} \cr
\end{aligned}
\]
\[
\begin{aligned}
\genm\mn13&=-{\frac {85}{1152}}\,{t}^{3}-{\frac {83}{13824}}\,{t}^{4} \cr
\genm\mn23&=-{\frac {23}{1024}}\,\sqrt {2}{t}^{3}-{\frac {367}{165888}}\,\sqrt {2}{t}^{4} \cr
\genm\mn33&=1+{\frac {31}{192}}\,{t}^{3}+{\frac {595}{4608}}\,{t}^{4}\cr
\genm\mn43&={\frac {89}{9216}}\,\sqrt {2}{t}^{3}-{\frac {287}{110592}}\,\sqrt {2}{t}^{4} \cr
\end{aligned}
\]
\[
\begin{aligned}
\genm\mn14&={\frac {323}{9216}}\,\sqrt {2}{t}^{3}+{\frac {173}{36864}}\,\sqrt {2}{t}^{4} \cr
\genm\mn24&={\frac {85}{1152}}\,{t}^{3}+{\frac {517}{18432}}\,{t}^{4} \cr
\genm\mn34&={\frac {89}{9216}}\,\sqrt {2}{t}^{3}-{\frac {287}{110592}}\,\sqrt {2}{t}^{4} \cr
\genm\mn44&=1+{\frac {95}{768}}\,{t}^{3}+{\frac {21781}{147456}}\,{t}^{4} \cr
\end{aligned}
\]
so $$U(t)U(t)^*=\genm\mn{}{}+\oaf3=I+\oaf3$$\hfill\qed
\end{proof}

\begin{lemma}\label{uinvu}$U(t)^{-1}=U(t)^*+\oaf3$\end{lemma}
\begin{proof}
\begin{align*}
  &U(t)U(t)^*=I+\oaf3 \cr
  \implies\qquad&U(t)^*=U(t)^{-1}\left(I+\oaf3\right) \cr
  \implies\qquad&U(t)^{-1}=U(t)^*\left(I+\oaf3\right)^{-1} \cr
\end{align*}
since $\oaf3^n\as {n\to\infty}{40} \bracketmatrix0_{4\times4}$ for small $t$,
\begin{align*}
\bigg(I+&\oaf3\bigg)^{-1}\cr
&=I-\oaf3+\oaf3^2-\oaf3^3\pm\dots=I+\oaf3
\end{align*}
so
\begin{align*}
&U(t)^{-1}=U(t)^*\left(I+\oaf3\right) \cr
\implies\qquad&U(t)^{-1}=U(t)^*+U(t)^*\oaf3 \cr
\implies\qquad&U(t)^{-1}=U(t)^*+\oaf3 \cr
\end{align*}
\hfill\qed
\end{proof}

\bigskip
Now
\begin{align*}
&J(t)=\cr&\sidey\!\intcf\!\left||w_1|^2-|w_2|^2\right|\!\!\cdot\!\!\left||w_3|^2-|w_4|^2\right|e^{-\angleenclose{U(t)^*D(t)U(t)w,w}}\,dw \cr
&=\!\sidey\!\intcf\!\left||w_1|^2-|w_2|^2\right|\!\!\cdot\!\!\left||w_3|^2-|w_4|^2\right|e^{-\angleenclose{D(t)U(t)w,U(t)w}}dw \cr
\end{align*}
Make the substitution
\begin{align*}
z&:=U(t)w\cr
z_i&=\sum_j\uu ijw_j \cr
w_i&=\sum_ju^{ij}z_j=\sum_j\big(\uu ji+\bigoo3\big)z_j \cr
dw&=\det(U(t))\cdot dz=\big(1+\bigoo3\big)\,dz=\bigg[dz+\bigoo3\,dz\bigg]
\end{align*}
In the following, only the properties of the $t^0$, $t^1$, and $t^2$ terms of the $u_{ji}$ are used so, for the sake of readability, ``$u_{ji}$'' will always be written in place of ``$\,u_{ji}+\bigoo3$''.

\begin{align}
\label{seereal}
&J(t)= \cr
&\sidey\intcf \left|\lsw1-\lsw2\right|\cdot \left|\lsw3-\lsw4\right|e^{-\angleenclose{D(t)U(t)w,U(t)w}}\,dw \cr
\end{align}
\begin{align*}
  &=\sidey\intcf {\scriptstyle\left|\usw1-\usw2\right|}\cr
&\hskip11em\cdot {\scriptstyle\left|\usw3-\usw4\right|}e^{-\angleenclose{D(t)z,z}}\,{\scriptstyle(dz+\bigoo3\,dz)} \cr
  &=\sideshort\intcf {\scriptstyle\left|\usw1-\usw2\right|}\cr
&\hskip11em\cdot \left|\scriptstyle\usw3-\usw4\right|\eeeehack6\,dz \cr
  &=\sideshort\intcf {\scriptstyle\left|\usw1-\usw2\right|}\cr
&\hskip11em\cdot \left|\scriptstyle\usw3-\usw4\right|\eeeehack8\,dz \cr
\end{align*}
Make the substitution
\begin{equation*}
w_j:=\sqrt{\lll j}z_j\quad\Longrightarrow\quad z_j={w_j\over\sqrt{\lll j}}\quad\hbox{and}\quad dw_j=\sqrt{\lll j}\,dz_j
\end{equation*}
so
\begin{align*}
dz&=\frac i2dz_1\w\dzb_1\w\dots\w\frac i2 dz_4\w\dzb_4=\frac i2{dw_1\over\sqrt{\lll1}}\w{\dwb_1\over\sqrt{\lll1}}\w\dots\w\frac i2{dw_4\over\sqrt{\lll4}}\w{\dwb_4\over\sqrt{\lll4}} \cr
&={1\over\lll1\lll2\lll3\lll4}\frac i2 dw_1\w\dwb_1\w\dots\w\frac i2dw_4\w\dwb_4 \cr
&={dw\over\lll1\lll2\lll3\lll4}
\end{align*}
so, since Lemma \ref{uustarI} and (\ref{detexpan}) imply
$$
[\det\Lambda(t)]\lll1\lll2\lll3\lll4=t^{20}+\bigoo{23}
$$
\begin{align*}
&J(t)=\cr
&\leftyyy\!\!\intcf\!{\scriptstyle\left|\usww1\!\!\!-\usww2\right|}\cr
&\hskip11em\cdot{\scriptstyle\left|\usww3\!\!\!-\usww4\right|}\eeeehack9\!dw \cr
&={t^{30}+\bigoo{33}\over\scriptstyle\pi^6\!\det\!A[t^{20}+\bigoo{23}]}\intcf {\scriptstyle \left| \usww1-\usww2 \right|}\cr
&\hskip11em\cdot   {\scriptstyle \left| \usww3-\usww4 \right|}\eeeehack4\,dw \cr
&={t^{10}+\bigoo{13}\over\scriptstyle\pi^6\det A(1+\bigoo3)}\intcf {\scriptstyle \left| \usww1-\usww2 \right|}\cr
&\hskip11em\cdot   {\scriptstyle \left| \usww3-\usww4 \right|}\eeeehack4\,dw \cr
&={t^{10}+\bigoo{13}\over\pi^6\det A(t)}\intcf {\scriptstyle \left| \usww1-\usww2 \right|}\cr
&\hskip11em\cdot   {\scriptstyle \left| \usww3-\usww4 \right|}\eeeehack4\,dw \cr
&={1+\bigoo3\over\pi^6\det A(t)}\intcf t^5{\scriptstyle \left| \usww1-\usww2 \right|}\cr
&\hskip11em\cdot   t^5{\scriptstyle \left| \usww3-\usww4 \right|}\eeeehack4\,dw \cr
&={1+\bigoo3\over\pi^6\det A(t)}\intcf {\scriptstyle \left| \uswww1-\uswww2 \right|}\cr
&\hskip11em\cdot   {\scriptstyle \left| \uswww3-\uswww4 \right|}e^{-\angleenclose{w,w}}\,dw\cr
\end{align*}

Now
$$
\begin{aligned}
&\brac11=-{\sqrt{30}\over8640}{t}^{\frac72}-{7\sqrt{30}\over69120}t^{\frac92}+\bigoo{11\over2} \cr
&\brac12=-{\sqrt{30}\over720}t^\frac52-{\sqrt{30}\over1152}t^\frac72+\bigoo{\frac92} \cr
&\brac13={\sqrt{30}\over8640}t^\frac72+{\sqrt{30}\over23040}t^\frac92+\bigoo{11\over2} \cr
&\brac14={\sqrt{30}\over720}t^\frac52+{\sqrt{30}\over5760}t^\frac72+\bigoo{\frac92} \cr
\end{aligned}
$$
$$
\begin{aligned}
&\brac21=-{\sqrt6\over12}t-{5\sqrt6\over96}t^2+\bigoo3\cr
&\brac22={\sqrt6\over12}t-{\sqrt6\over96}t^2+\bigoo3 \cr
&\brac23=-{\sqrt6\over12}t-{\sqrt6\over96}t^2+\bigoo3 \cr
&\brac24={\sqrt6\over12}t+{\sqrt6\over32}t^2+\bigoo3 \cr
\end{aligned}
$$
and
$$
\begin{aligned}
&\brac31={\sqrt6\over2}t^\frac12+{\sqrt6\over16}t^\frac32+\bigoo{\frac52} \cr
&\brac32={\sqrt6\over12}t^\frac32+{67\sqrt6\over576}t^\frac52+\bigoo{\frac72} \cr
&\brac33=-{\sqrt6\over2}t^\frac12-{\sqrt6\over16}t^\frac32+\bigoo{\frac52} \cr
&\brac34={\sqrt6\over6}t^\frac32+{121\sqrt6\over576}t^\frac52+\bigoo{\frac72} \cr
\end{aligned}
$$
$$
\begin{aligned}
&\brac41={\sqrt6\over3}-{7\sqrt6\over72}t-{473\sqrt6\over3456}t^2+\bigoo3 \cr
&\brac42={\sqrt6\over3}-{\sqrt6\over72}t-{29\sqrt6\over3456}t^2+\bigoo3\cr
&\brac43={\sqrt6\over3}+{5\sqrt6\over72}t+{799\sqrt6\over3456}t^2+\bigoo3 \cr
&\brac44={\sqrt6\over3}+{11\sqrt6\over72}t+{91\sqrt6\over3456}t^2+\bigoo3 \cr
\end{aligned}
$$
\newcommand{\roo}[2]{{\sqrt{#1}\over#2}}
\newcommand{\tto}[2]{t^{#1\over#2}}
so
\begin{align*}
&J(t)=\cr
&{1+\bigoo3\over\pi^6\det A(t)}\intcf|\al\an_t(w)\al\an_t(\wb)-\be\an_t(w)\be\an_t(\wb)|\cr
&\hskip11em\cdot|\ga\an_t(w)\ga\an_t(\wb)-\de\an_t(w)\de\an_t(\wb)|\cdot e^{-\angleenclose{w,w}}\,dw
\end{align*}
where
\begin{align*}
\al\an_t(w)=&\phantom+\bi72w_1\cr
&+\left( -{\sqrt6\over12}t-{5\sqrt6\over96}t^2+\bigoo3\right)w_2\cr
&+\left( \roo62t^\frac12+\roo6{16}t^\frac32+\bi52\right)w_3\cr
&+\left( \roo63-{7\sqrt6\over72}t-{473\sqrt6\over3456}t^2+\bigoo3\right)w_4\cr
\end{align*}\begin{align*}
\be\an_t(w)=&\phantom+\bi52w_1\cr
&+\left( \roo6{12}t-\roo6{96}t^2+\bigoo3\right)w_2\cr
&+\left( \roo6{12}t^\frac32+\bi52\right)w_3\cr
&+\left( \roo63-\roo6{72}t-{29\sqrt6\over3456}t^2+\bigoo3\right)w_4\cr
\end{align*}\begin{align*}
\ga\an_t(w)=&\phantom+\bi72w_1\cr
&+\left( -\roo6{12}t-\roo6{96}t^2+\bigoo3\right)w_2\cr
&+\left( -\roo62t^\frac12-\roo6{16}t^\frac32+\bi52\right)w_3\cr
&+\left( \roo63+{5\sqrt6\over72}t+{799\sqrt6\over3456}t^2+\bigoo3\right)w_4\cr
\end{align*}\begin{align*}
\de\an_t(w)=&\phantom+\bi52w_1\cr
&+\left( \roo6{12}t+\roo6{32}t^2+\bigoo3\right)w_2 \cr
&+\left( \roo66t^\frac32+\bi52\right)w_3 \cr
&+\left( \roo63+{11\sqrt6\over72}t+{91\sqrt6\over3456}t^2+\bigoo3\right)w_4\cr
\end{align*}
\global\advance\ac by 1
So
\begin{align*}
&J(t)={1+\bigoo3\over\pi^6\det A}\!\!\intcf
\!\!\bigg|\al\an+\be\an\sqrt t+\ga\an t+\de\an t^\frac32+{\sume1}\bigg|\cr
&\hskip8em\cdot
\global\advance\ac by 1
\bigg|\al\an+\be\an\sqrt t+\ga\an t+\de\an t^\frac32+{\sume2}\bigg|
\,
{e e^{-\angleenclose{w,w}}}\,dw
\end{align*}
where
\begin{align*}
  \al2(t)&=\roo63w_4\roo63\wb_4-\roo63w_4\roo63\wb_4=0\cr
  &=\al3(t) \cr
\end{align*}\begin{align*}
  \be2(t)&=\roo62w_3\roo63\wb_4+\roo63w_4\roo62\wb_3=\left( {w_3}\wb_4 +\wb_3{w_4}   \right)=2\realpart(w_3\wb_4)\cr
\end{align*}\begin{align*}
  \be3(t)&=-\roo62w_3\roo63\wb_4+\left(-\roo62\right)\wb_3\roo63w_4=-\left( {w_3}\wb_4+\wb_3{w_4}   \right)\cr
&=-2\realpart(w_3\wb_4) =-\be2(t) \cr
\end{align*}\begin{align*}
  \ga2(t)&= -\roo6{12}w_2\roo63\wb_4+\roo62w_3\roo62\wb_3  +\roo63w_4\left(-\roo6{12}\right)\wb_2\cr
&\phantom{{}=}+\roo63w_4(-{7\sqrt6\over72})\wb_4  -{7\sqrt6\over72}w_4\roo63\wb_4 -\roo6{12}w_2\roo63\wb_4  \cr&\phantom{{}=}-\roo63w_4\roo6{12}\wb_2-\roo63w_4\left(-\roo6{72}\right)\wb_4  +\roo6{72}w_4\roo63\wb_4\cr
  &=  \frac16\left( -2{w_2}\wb_4 -2\wb_2{w_4} -2{w_4}\wb_4 + 9{w_3}\wb_3  \right)\cr
&=\frac16(-4\realpart(w_2\wb_4)+9|w_3|^2-2|w_4|^2)\cr
\end{align*}\begin{align*}
  \ga3(t)=& -\roo6{12}w_2\roo63\wb_4+(-\roo62)w_3(-\roo62)\wb_3 +\roo63w_4(-\roo6{12})\wb_2\cr
&+\roo63w_4{5\sqrt6\over72}\wb_4 +{5\sqrt6\over72}w_4\roo63\wb_4 -\roo6{12}w_2\roo63\wb_4 \cr
&-\roo63w_4\roo6{12}\wb_2-\roo63w_4{11\sqrt6\over72}\wb_4  -{11\sqrt6\over72}w_4\roo63\wb_4  \cr
  =&\frac16\left( -2{w_2}\,\wb_4 -2\wb_2w_4 -2{w_4}\,\wb_4 +9{w_3}\,\wb_3 \right)\cr
  =&\frac16(-4\realpart(w_2\wb_4)+9|w_3|^2-2|w_4|^2) \cr
  &=\ga2(t) \cr
\end{align*}\begin{align*}
  \de2(t)&=-{1\over12}(3\wb_2w_3+3w_2\wb_3+4\wb_3w_4+4w_3\wb_4)t^\frac32\cr
&=-{1\over12}(6\realpart(w_2\wb_3)+8\realpart(w_3\wb_4))\cr
\end{align*}\begin{align*}
  \de3(t)&={1\over12}(3w_2\wb_3+3\wb_2w_3-8w_3\wb_4-8\wb_3w_4)t^\frac32\cr
&={1\over12}(6\realpart(w_2\wb_3)-16\realpart(w_3\wb_4)) \cr
\end{align*}\begin{align*}
  {}^1\varepsilon_{jk}(t)&=\bigoo2 \cr
  {}^2\varepsilon_{jk}(t)&=\bigoo2 \cr
\end{align*}
notice $\al2,\al3,\be2,\be3,\ga2,\ga3,\de2,\de3\in\mathbb{R}$.  This implies that since
$$\bigg(\alpha_2+\beta_2\sqrt t+\gamma_2t+\delta_2t^\frac32+\sume1\bigg)\in\mathbb{R}$$
and
$$\bigg(\alpha_3+\beta_3\sqrt t+\gamma_3t+\delta_3t^\frac32+\sume2\bigg)\in\mathbb{R}$$
by
(\ref{seereal}),\quad$\bigg(\sume1\bigg)$ and $\bigg(\sume2\bigg)$ must be real, as well.
\goodbreak
so
\begin{align*}
&J(t)=\cr
&{1+\bigoo3\over\pi^6\det A(t)}\intcf|\be2\sqrt t+\ga2t+\de2t^\frac32+{\sume1}|\cr
&\hskip9em\cdot|-\be2\sqrt t+\ga2t+\de3t^\frac32+{\sume2}|\,{e^{-\angleenclose{w,w}}\,dw} \cr
=&{1+\bigoo3\over\pi^6\det A}   \intcf    \bigg|\kern-1.8em\smash{\overbrace{[{-\be2^2}]}^{\displaystyle\phantom{{}<0}\alpha_4(w)<0}}\kern-1.5em t+[0]t^\frac32+\overbrace{[{\be2\de3+\ga2^2-\de2\be2}]}^{\displaystyle\beta_4(w)}t^2\cr
&\hskip7em+\underbrace{\sume3}_{\displaystyle\SS3} + \underbrace{\sumew4}_{\displaystyle\SS4} \bigg|{e^{-\angleenclose{w,w}} dw} \cr
\end{align*}
where
\begin{align*}
\al4&=-(\be2)^2=-(2\realpart(w_3\wb_4))^2=-4\realpart(w_3\wb_4)^2 \cr
\end{align*}\begin{align*}
\be4&=\be2\de3+\ga2^2-\de2\be2 \cr
&={1\over6}\realpart(w_3\wb_4)\bigg(6\realpart(w_2\wb_3)-16\realpart(w_3\wb_4)\bigg) \cr
&\phantom{={}}+{1\over36}\bigg(-4\realpart(w_2\wb_4)+9|w_3|^2-2|w_4|^2\bigg)^2 \cr
&\phantom{={}}+{1\over6}\realpart(w_3\wb_4)\bigg(6\realpart(w_2\wb_3)+8\realpart(w_3\wb_4)\bigg) \cr
&= 2  \realpart  ( {w_3} {\wb_4}  )  \realpart  ( {w_2} {\wb_3}  ) -\frac43   (  \realpart  ( {w_3} {\wb_4}  )  ) ^{2}+\frac49   (  \realpart  ( {w_2} {\wb_4}  ) ) ^{2}\cr
&\phantom{={}}-2  \realpart  ( {w_2} {\wb_4} ) {w_3} {\wb_3} +\frac49   \realpart  ( {w_2}  {\wb_4}  ) {w_4} {\wb_4} +\frac94  {{w_3}}^{2}  {\wb_3^2}- {w_3} {\wb_3} {w_4} {\wb_4} \cr
&\phantom{={}} +\frac19  {{w_4}}^{2}  {\wb_4^2}   \cr
&=2\realpart(w_2\wb_3)\realpart(w_3\wb_4)-\frac43\realpart(w_3\wb_4)^2 \cr
&\phantom{{}=}+\frac19\realpart(w_2\wb_4)\bigg(4\realpart(w_2\wb_4)-18|w_3|^2+4|w_4|^2\bigg) \cr
&\phantom{={}}+\frac94|w_3|^4-|w_3|^2|w_4|^2+\frac19|w_4|^4 \cr
&=2\realpart(w_2\wb_3)\realpart(w_3\wb_4)-\frac43\realpart(w_3\wb_4)^2\cr
&\phantom{{}=}+\frac19\realpart(w_2\wb_4)\bigg(4\realpart(w_2\wb_4)-18|w_3|^2+4|w_4|^2\bigg) \cr &\phantom{={}}+{1\over36}{\bigg(9|w_3|^2-2|w_4|^2\bigg)^2} \cr
\end{align*}\begin{align*}
{}^3\varepsilon_{jk}(t)&=\bigoo{5\over2} \cr
{}^4\varepsilon_{jk\ell m}(t)&={}^1\varepsilon_{jk}(t)\cdot{}^2\varepsilon_{\ell m}(t)=\bigoo4 \cr
\end{align*}
Again, notice that $\al4(w)$, $\be4(w)$, $\SS3$, and $\SS4$ are all real.

So
$$
J(t)={1+\bigoo3\over\pi^6\det A(t)}\intcf\left|\al4t+\be4t^2+\SS3+\SS4\right|\,e^{-\angleenclose{w,w}}\,dw
$$

The following lemma will help take the error terms out of the absolute value.
\newcommand{\fwt}{f(w,t)}
\newcommand{\eww}{\e(t)g(w)}
\newcommand{\ww}{\left|g(w)\right|}
\begin{lemma}\label{shrink}  For any $n>1$, any continuous complex valued $\fwt$ and $g(w)$, and $\e(t)=\bigoo n$
$$
|\fwt|=\left|\fwt-\e(t)g(w)\right|+\bigoo n|g(w)|
$$
\end{lemma}

\begin{proof}
$$
\bigg||\fwt|-|\fwt-\eww|\bigg|\leq\bigg|\eww\bigg|=|\e(t)|\,\ww
$$
so
$$
\underbrace{-|\e(t)|}_{\displaystyle\bigoo n}\cdot\,\ww\leq|\fwt|-|\fwt-\eww|\leq\underbrace{|\e(t)|}_{\displaystyle\bigoo n}\cdot\,\,\ww
$$
so
$$
|\fwt|-|\fwt-\eww|=\bigoo n\ww
$$
so
$$
|\fwt|=|\fwt-\eww|+\bigoo n\ww
$$
\hfill\qed
\end{proof}

Applying Lemma \ref{shrink} repeatedly to $|\al4t+\be4t^2+\SS3+\SS4|$ and $\eps3jkw_j\wb_k$ says that
\begin{align*}
&J(t)=\cr
&{1+\bigoo3\over\pi^6\det A(t)}\intcf\bigg(\left|\al4t+\be4t^2+\SS4\right|+\sum_{jk}\bigoo{5\over2}|w_j\wb_k|\bigg)\,e^{-\angleenclose{w,w}}\,dw
\end{align*}
and applying Lemma \ref{shrink} repeatedly to $|\al4t+\be4t^2+\SS4|$ and\break${}^4\varepsilon_{jk\ell m}(t)w_j\wb_kw_\ell \wb_m$ says
\begin{align*}
&J(t)={1+\bigoo3\over\pi^6\det A(t)}\intcf\bigg(\left|\al4t+\be4t^2\right|+\sum_{jk}\bigoo{5\over2}|w_j\wb_k| \cr
&\hskip14em+\sum_{jk\ell m}\bigoo4|w_j\wb_kw_\ell\wb_m|\bigg)\,e^{-\angleenclose{w,w}}\,dw
\end{align*}
Finally, applying Lemma \ref{shrink} to $|\al4t+\be4t^2|$ and $\be4t^2$ says
\begin{align*}
J(t)=&{\scriptstyle1+\bigoo3\over\scriptstyle\pi^6\det A(t)}\intcf\bigg(|\al4t|+|\be4|\bigoo2+\sum_{jk}\bigoo{5\over2}|w_j\wb_k|\cr
&\hskip11em+\sum_{jk\ell m}\bigoo4|w_j\wb_kw_\ell\wb_m|\bigg)\edub \cr
=&\phantom{{}+{}}{1+\bigoo3\over\pi^6\det A(t)}\left(-t\intcf\al4\edub+\bigoo2\intcf|\be4|\edub \right) \cr &+{1+\bigoo3\over\pi^6\det A(t)}\bigg(\sum_{jk}\bigoo{5\over2}\intcf|w_j\wb_k|\edub\cr
&\hskip11em+\sum_{jk\ell m}\bigoo4\intcf|w_j\wb_kw_\ell\wb_m|\edub\bigg)\cr
=&\phantom{{}+{}}{1+\bigoo3\over\pi^6\det A(t)}\bigg(-t (-2\pi^4) +\bigoo2 \big[\hbox{finite}\big] +\sum_{jk}\bigoo{5\over2} \big[\hbox{finite}\big]\cr
&\hskip11em+\sum_{jk\ell m}\bigoo4 \big[\hbox{finite}\big] \bigg)\cr
=&\phantom{{}+{}}{1+\bigoo3\over\pi^2\det A(t)}\left(2t+\bigoo2 +\bigoo{5\over2}  +\bigoo4  \right)\cr
=&\phantom{{}+{}}{1+\bigoo3\over\pi^2\det A(t)}\bigg(2t+\bigoo2\bigg)\cr
=&\phantom{{}+{}}{2t+\bigoo2+\bigoo3+\bigoo4+\bigoo5\over\pi^2\det A(t)}\cr
=&\phantom{{}+{}}{2t+\bigoo2\over\pi^2\det A(t)}\cr
=&\phantom{{}+{}}{1\over\pi^2}{2t+\bigoo2\over3t+\bigoo2} \cr
=&\phantom{{}+{}}{2\over3\pi^2}+\bigo t
\end{align*}
so 
$$J(r)={2\over3\pi^2}+\bigo{r^2}$$  i.e.
$$
J(r)\as{r\searrow0}{30}{2\over3\pi^2}
$$

\section{Appendix}

\noindent or \it Where did $U(t)$ and $D(t)$ come from?\rm

$D(t)$ and $U(t)$ naturally arise when calculating $J(t)$ by diagonalizing $Y(t)$.
\begin{align*}
  J(t)&=\lefty\intcf\left||\h1|^2-|x_1|^2\right|\cdot\left||\h2|^2-|x_2|^2\right|\cr
&\hskip20em\cdot e^{-\angleenclose{ Y(t)(t^{-\frac52}v),(t^{-\frac52}v)}}\,dv \cr
  &=\lefty\intcf t^5\left||w_1|^2-|w_2|^2\right|\cdot t^5\left||w_3|^2-|w_4|^2\right|\cr
&\hskip20em\cdot e^{-\angleenclose{Y(t)w,w}}\,t^{20}dw \cr
  &=\sidey\intcf \left||w_1|^2-|w_2|^2\right|\cdot \left||w_3|^2-|w_4|^2\right|\cr
&\hskip20em\cdot e^{-\angleenclose{Y(t)w,w}}\,dw \cr
  &=\sidey\intcf \left||w_1|^2-|w_2|^2\right|\cdot \left||w_3|^2-|w_4|^2\right|\cr
&\hskip20em\cdot e^{-\angleenclose{U(t)^*D(t)U(t)w,w}}\,dw \cr
  &=\sidey\intcf \left||w_1|^2-|w_2|^2\right|\cdot \left||w_3|^2-|w_4|^2\right|\cr
&\hskip20em\cdot e^{-\angleenclose{D(t)U(t)w,U(t)w}}\,dw \cr
\end{align*}

where
$$
D(t)=\bracketmatrix{\lll1(t)&0&0&0\cr0&\lll2(t)&0&0\cr0&0&\lll3(t)&0\cr0&0&0&\lll4(t)}
$$
with the $\lll i(t)$ being the eigenvalues of $Y(t)$ and
\newcommand{\hor}[1]{
\leaders\hrule height3pt depth-2pt width1cm \hskip3em\ #1\ \leaders\hrule height3pt depth-2pt width1cm \hskip3em
}
$$
U(t)=\bracketmatrix{
\hor{v_1(t)}\cr
\hor{v_2(t)}\cr
\hor{v_3(t)}\cr
\hor{v_4(t)}\cr
}
$$
with the $v_i(t)$ being the associated normalized eigenvectors, making $U(t)$ real orthogonal.

So expanding the $\lll i$ and $v_i$ in $t$ gives an expansion for $J(t)$.  Maple outputs a 100 megabyte file for each eigenvalue when asked $\tt eigenvalues(Y(t))$ directly and crashes when asked to find an expansion for any individual eigenvalue.  However, because the matrix is $4\times4$, the eigenvalues and eigenvectors can be calculated algebraically by applying the quartic formula \cite{cardano} to $Y(t)$'s characteristic polynomial.  Then these algebraic expressions can be expanded by maple.

$$
Y(t)=:\bracketmatrix{\f1(t)&\f2(t)&\f3(t)&\f4(t)\cr\f5(t)&\f6(t)&\f7(t)&\f8(t)\cr\f9(t)&\f{10}(t)&\f{11}(t)&\f{12}(t)\cr\f{13}(t)&\f{14}(t)&\f{15}(t)&\f{16}(t)\cr}
$$

$ \det(Y(t)-xI)=$
$$
x^4
+\underbrace{\left(
\begin{aligned}
-\f6 \cr-\f1 \cr-\f{11}\cr-\f{16}
\end{aligned}
\right)}_{\F3(t)}x^3
+\underbrace{\left(
\begin{aligned}
-\f9  \f3 \cr +\f6  \f{16}\cr -\f{13} \f4 \cr +\f6  \f{11}\cr -\f5  \f2 \cr -\f{14} \f8 \cr +\f{11} \f{16}\cr -\f{12} \f{15}\cr -\f{10} \f7 \cr +\f1  \f6 \cr +\f1  \f{11}\cr +\f1  \f{16}
\end{aligned}
\right)}_{\F2(t)}x^2
+\underbrace{\left(
\begin{aligned}
-\f1  \f6  \f{16}\cr +\f{13} \f6  \f4 \cr -\f1  \f6  \f{11}\cr -\f5  \f{14} \f4 \cr +\f1  \f{10} \f7 \cr -\f{13} \f3  \f{12}\cr +\f{13} \f4  \f{11}\cr +\f5  \f2  \f{11}\cr -\f9  \f2  \f7 \cr +\f9  \f3  \f{16}\cr -\f9  \f4  \f{15}\cr -\f5  \f{10} \f3 \cr +\f5  \f2  \f{16}\cr -\f1  \f{11} \f{16}\cr +\f1  \f{12} \f{15}\cr -\f{13} \f2  \f8 \cr +\f9  \f6  \f3 \cr +\f1  \f{14} \f8 \cr -\f6  \f{11} \f{16}\cr +\f6  \f{12} \f{15}\cr +\f{10} \f7  \f{16}\cr -\f{10} \f8  \f{15}\cr -\f{14} \f7  \f{12}\cr +\f{14} \f8  \f{11}
\end{aligned}
\right)}_{\F1(t)}x
+\underbrace{\left(
\begin{aligned}
-\f5  \f2  \f{11} \f{16}\cr -\f5  \f2  \f{12} \f{15}\cr -\f5  \f{10} \f3  \f{16}\cr +\f5  \f{10} \f4  \f{15}\cr +\f5  \f{14} \f3  \f{12}\cr -\f5  \f{14} \f4  \f{11}\cr -\f9  \f2  \f7  \f{16}\cr +\f9  \f2  \f8  \f{15}\cr +\f9  \f6  \f3  \f{16}\cr -\f9  \f6  \f4  \f{15}\cr -\f9  \f{14} \f3  \f8 \cr +\f9  \f{14} \f4  \f7 \cr +\f{13} \f2  \f7  \f{12}\cr -\f{13} \f2  \f8  \f{11}\cr -\f{13} \f6  \f3  \f{12}\cr +\f{13} \f6  \f4  \f{11}\cr +\f{13} \f{10} \f3  \f8 \cr -\f{13} \f{10} \f4  \f7 \cr -\f1  \f6  \f{11} \f{16}\cr +\f1  \f6  \f{12} \f{15}\cr +\f1  \f{10} \f7  \f{16}\cr -\f1  \f{10} \f8  \f{15}\cr -\f1  \f{14} \f7  \f{12}\cr +\f1  \f{14} \f8  \f{11}
\end{aligned}
\right)}_{\F0(t)}
$$
Because $Y$ is hermitian, it's eigenvalues will all be real.  In fact, because $Y$ is positive definite, they will all be positive.  In the calculation below, things can be complex (e.g. $R$) but all of the imaginary parts will go away by the end.  
\begin{align*}
\F3(t)&=-8640+\bigo t \cr
\F2(t)&=51840\,{t}^{3}+\bigoo4  \cr
\F1(t)&=-17280\,{t}^{7}+\bigoo8  \cr
\F0(t)&=6480\,{t}^{12}+\bigoo{13} \cr
\end{align*}

Solving the general quartic $x^4+Bx^3+Cx^2+Dx+E=0$ requires some simplifying definitions and a few choices.
\begin{align*}
\alpha&:=-{3B^2\over8}+C=-27993600+\bigo t \qquad\negsmall\cr
\beta&:={B^3\over8}-{BC\over2}+D=-80621568000+\bigo t \qquad\negsmall\cr
\gamma&:=-{3B^4\over256}+{CB^2\over16}-{BD\over4}+E=-65303470080000+\bigo t \cr
P&:=-{\alpha^2\over12}-\gamma=-223948800\,{t}^{6}+\bigoo7 \cr
Q&:=-{\alpha^3\over108}+{\alpha\gamma\over3}-{\beta^2\over8}=-1289945088000\,{t}^{9}+\bigoo{10}  \cr
\end{align*}\begin{align*}
R&=-\frac Q2\pm\sqrt{{Q^2\over4}+{P^3\over27}}\cr&\cr&=644972544000\,{t}^{9}+\bigoo{10}\pm\underbrace{\sqrt{-8666449635704832000000\,{t}^{20}+\bigoo{21} }}_{i\bigoo{10}}
\cr &\hbox{(Choose either root.)} \cr
\end{align*}\begin{align*}
U&=\root3\of R&\hbox{(Choose any of the three roots.)} \cr
y&:=-\frac56\alpha+U-{P\over3U} \cr
W&:=\sqrt{\alpha+2y} &{\hbox{(Both roots come up in $\pm_s$)}}\cr
\end{align*}
With those definitions, the roots should be
$$
-\frac B4+{\pm_s W\pm_t\sqrt{-\left(3\alpha+2y\pm_s{2\beta\over W}\right)}\over2}
$$
where the $\pm_s$'s are dependent and the $\pm_t$ is independent.
\rbreak
Maple can expand $\alpha$, $\beta$, and $B$ easily.  We need to carefully intervene to get it to expand $W$, $y$, and ${2\beta\over W}$ which require $R$ and $U$.
\rbreak
\begin{align*}
&\kern-1em-{Q^2\over4}-{P^3\over27}=\cr
&{\scriptstyle8666449635704832000000}\,{t}^{20}\!+\cdots+\!{\frac {\scriptscriptstyle563370492281772061551028318944792605312337}{7662929083743683750}}{t}^{40}\!+\!\bigoo{41}
\end{align*}
\rbreak
\begin{align*}
   \sqrt{-{Q^2\over4}-{P^3\over27}}=&\sqrt{\aaa t^{20}+\dots+\bigoo{41}}\cr
 =&\sqrt{\aaa}\,t^{10}\cdot\sqrt{1+\dots+\bigoo{21}}  \underx\cr
 =&\sqrt{\aaa}\,t^{10}\cdot\left(1+\frac12x+\dots+\bigo{x^{21}}\right)\oversqrtx\cr
 =&\sqrt{\aaa}\,t^{10}+\dots+\bigoo{31}
\end{align*}
\rbreak
$$
-{Q\over2}=644972544000\,{t}^{9}+\dots+\bigoo{30}
$$
\rbreak
\begin{align*}
&{\sqrt{-{Q^2\over4}-{P^3\over27}}\over-\frac Q2}=\cr
&{\sqrt\aaa\,t^{10}+\dots+\bigoo{31}\over\bbb\,t^9+\dots+\bigoo{30}}\cr
=&{1\over{\scriptstyle\bbb}\,t^9}{1\over1+\dots+\bigoo{21}\underx}\cr
&\hskip11em\cdot\big(\sqrt{\scriptstyle\aaa}\,t^{10}+\dots+\bigoo{31}\big) \cr
\cr
=&{1\over{\scriptstyle\bbb}\,t^9}\big(1-x+\dots+\bigo{x^{21}}\big)\overgeo\cr
&\hskip11em\cdot\big(\sqrt{\scriptstyle\aaa}\,t^{10}+\dots+\bigoo{31}\big)\cr
=&{1\over\bbb\,t^9}\big(1+\dots+\bigoo{21}\big)\cr
&\hskip11em\cdot\big(\sqrt{\scriptstyle\aaa}\,t^{10}+\dots+\bigoo{31}\big)\cr
=&{1\over\bbb\,t^9}\big(\sqrt\aaa\,t^{10}+\dots+\bigoo{31}\big)\cr
=&{\sqrt3\over12}\,t+\dots+\bigoo{22}
\end{align*}
\rbreak
Choose $R$ in the I quadrant, namely
\begin{align*}
R:=&-{Q\over2}+i\sqrt{-{Q^2\over4}-{P^3\over27}}\cr
  =&\bbb\,t^9+\dots+\bigoo{30}\cr
&+i\big(\sqrt\aaa\,t^{10}+\dots+\bigoo{31}\big) \cr
  =&\rr e^{i\trr}
\end{align*}
where
$$
\rr=\sqrt{{-P^3\over27}}\quad\hbox{and}\quad\trr=\arctan\left(\sqrt{-{Q^2\over4}-{P^3\over27}}\over-{Q\over2}\right)
$$
\rbreak
Choose $U$ in the I quadrant, namely
$$
U:=\root3\of R=\root3\of{\rr}e^{i\frac13\trr}=\ru e^{i\tu}
$$
where
$$
\tu=\frac13\arctan\left({\sqrt{-{Q^2\over4}-{P^3\over27}}\over-{Q\over2}}\right)\quad\hbox{and}\quad\ru=\left({(-P)^3\over27}\right)^{\frac16}={\sqrt{-P}\over\root6\of{27}}
$$
\rbreak
\begin{align*}
y&=-\frac56\alpha+U-{P\over3U}=-\frac56\alpha+\ru e^{i\tu}-{P\over3\ru e^{i\tu}}\cr
&=-\frac56\alpha+\ru e^{i\tu}-{Pe^{-i\tu}\over3\ru}\cr
&=-\frac56\alpha+{\ru^2e^{i\tu}-{P\over3}e^{-i\tu}\over\ru}\cr
&=-\frac56\alpha+{\ru^2(\cos\tu+i\sin\tu)-{P\over3}(\cos\tu-i\sin\tu)\over\ru}\cr
&=-\frac56\alpha+{\left(\ru^2-\frac P3\right)\over\ru}\cos\tu+i\left[{\left(\ru^2+\frac P3\right)\over\ru}\right]\sin\tu\cr
&=-\frac56\alpha+\left[{\left({\sqrt{-P}\over\root6\of{27}}\right)^2-\frac P3\over\ru}\right]\cos\tu+i\left[{\left({\sqrt{-P}\over\root6\of{27}}\right)^2+\frac P3\over\ru}\right]\sin\tu\cr
&=
-\frac56\alpha+\left({{-P\over3}-\frac P3\over\ru}\right)\cos\tu+i\left({{-P\over3}+\frac P3\over\ru}\right)\sin\tu\cr
&=-\frac56-{2P\over3\ru}\cos\tu+i\begin{pmatrix}\cr&0&\cr\cr\end{pmatrix}\sin\tu\cr
&=-\frac56\alpha-{2P\over3\left({\sqrt{-P}\over\root6\of{27}}\right)}\cos\tu\cr
&=-\frac56\alpha+{2\over\sqrt3}\sqrt{-P}\cos\tu\cr
&=-\frac56\alpha+{2\over\sqrt3}\sqrt{-P}\cos\left(\frac13\arctan\left({\sqrt{-{Q^2\over4}-{P^3\over27}}\over-{Q\over2}}\right)\right)
\end{align*}
\rbreak
\begin{align*}
\arctan\left({\sqrt{-{Q^2\over4}-{P^3\over27}}\over-{Q\over2}}\right)&=\arctan\left({\sqrt3\over12}t+\dots+\bigoo{22}\right)\underxspec{9.2}{2}{8}{1}\cr
&=\left(x-\frac13x^2+\dots+\bigo{x^{22}}\right)\underarctan{11}2{10}{.9} \cr
&={\sqrt{3}\over12}t+\dots+\bigoo{22}
\end{align*}
so
$$
\frac13\arctan\left({\sqrt{-{Q^2\over4}-{P^3\over27}}\over-{Q\over2}}\right)={\sqrt3\over36}t+\dots+\bigoo{22}
$$
\rbreak
\begin{align*}
\cos\frac13&\arctan\left({\sqrt{-{Q^2\over4}-{P^3\over27}}\over-{Q\over2}}\right)=\cos\left({\sqrt3\over36}t+\dots+\bigoo{22}\right)\underxspec{9.2}281\cr
&=1+\dots+\bigo{x^{22}}\undercos{7.5}1{7.2}{.2}=1+\dots+\bigoo{22}
\end{align*}
\rbreak
\begin{align*}
y&=-\frac56\alpha+{2\over\sqrt3}\sqrt{-P}\cos\frac13\arctan\left({\sqrt{-{Q^2\over4}-{P^3\over27}}\over-{Q\over2}}\right) \cr
=&-\frac56\left(-\ddd+\dots+\bigoo{25}\right)\cr
&\hskip9em+{2\over\sqrt3}\left(\ccc t^3+...\bigoo{25}\right)\left(1+\dots+\bigoo{22}\right) \cr
=&{5\cdot\ddd\over6}+\dots+\bigoo{25}+{2\cdot\ccc\over\sqrt3}\,t^3+\dots+\bigoo{25}\cr
=&{5\cdot\ddd\over6}+\dots+\bigoo{25}
\end{align*}
\rbreak
\begin{align*}
W&=\sqrt{\alpha+2y}\cr
&=\sqrt{(-\ddd+\dots+\bigoo{25})+2\left({5\cdot\ddd\over6}+\dots+\bigoo{25}\right)} \cr
=&\sqrt{\frac23\cdot\ddd+\dots+\bigoo{25}}=\sqrt{\frac23\cdot\ddd}\sqrt{1+\dots+\bigoo{25}}\underx \cr
=&\sqrt{\frac23\cdot\ddd}\underbrace{\left(1+\dots+\bigo{x^{25}}\right)}_{\sqrt{1+x}}=\sqrt{\frac23\cdot\ddd}\left(1+\dots+\bigoo{25}\right) \cr
&=\sqrt{\frac23\cdot\ddd}+\dots+\bigoo{25}=4320+\dots+\bigoo{25}\cr
\end{align*}
\rbreak
\begin{align*}
{2\beta\over W}&={-{\scriptstyle\eee}+\dots+\bigoo{25}\over4320+\dots+\bigoo{25}}\cr
&=(-{\scriptstyle\eee}+\dots+\bigoo{25})\cdot{1\over4320}\cdot{1\over1+\dots+\bigoo{25}}\underxspec{5.5}2{5.5}1 \cr
\cr
&={1\over4320}(-\eee+\dots+\bigoo{25})(1+\dots+\bigo{x^{25}})\overgeospec{-8}{2.5}{4.6}{1.6}{3.35}5{2.2}\cr
&=-{\eee\over4320}+\dots+\bigoo{25}
=-37324800+\dots+\bigoo{25}
\end{align*}
\rbreak
So
\begin{align*}
\lll1&:=\lambda_{+,+}=-{B\over4}+{W+\sqrt{-(3\alpha+2y+{2\beta\over W})}\over2}={-B+2\sqrt{-3\alpha-2y-{2\beta\over W}}\over4} \cr
&=\frac14\left(
{\begin{aligned}&-({-8640+\dots+\bigoo{25}})+{2W}\cr
&+2\sqrt{\begin{aligned}&\scriptstyle-3({\scriptscriptstyle\!-\!\ddd+\dots+\bigoo{25}})\cr
&-2({\scriptscriptstyle{5\cdot\ddd\over6}+\dots+\bigoo{25}})-({\scriptscriptstyle{-\eee\over4320}+\dots+\bigoo{25}})\end{aligned}}\end{aligned}}
\right) \cr
&=\frac14\left({\begin{aligned}&({8640+\dots+\bigoo{25}})+2({4320+\dots+\bigoo{25}})\cr
&+2\sqrt{74649600+\dots+\bigoo{25}}\end{aligned}}\right) \cr
&={17280+\dots+\bigoo{25}+2\cdot8640\sqrt{1+\dots+\bigoo{25}}\over4}\underxspec{5.5}0{5.5}1 \cr
&={17280+\dots+\bigoo{25}+17280(1+\dots+\bigo{x^{25}})\over4}\undersqrt8080\cr
&=8640+\dots+\bigoo{25}
\end{align*}
\rbreak
\begin{align*}
\lll2&:=\lambda_{+,-}={-B+2W-2\sqrt{-3\alpha-2y-{2\beta\over W}}\over4}\cr
&={17280+\dots+\bigoo{25}-17280(1+\dots+\bigoo{25})\over4} \cr
&=6t^3+\dots+\bigoo{25}
\end{align*}
\rbreak
\begin{align*}
\lll3&:=\lambda_{-,+}={-B-2W+2\sqrt{-3\alpha-2y+{2\beta\over W}}\over4}\cr
&={-B-2W+2\sqrt{\frac19t^8+\dots+\bigoo{25}}\over4} \cr
&={-B-2W+\frac23t^4\sqrt{1+\dots+\bigoo{17}}\over4}\underxspec6060\cr
&=\frac14\left[{-B-2W+\frac23t^4\underbrace{(1+\dots+\bigo{x^{17}})}_{\sqrt{1+x}}}\right] \cr
&={-B-2W+\frac23t^4+\dots+\bigoo{21}\over4}=\frac13t^4+\dots+\bigoo{21}
\end{align*}
\rbreak
\begin{align*}
\lll4&:=\lambda_{-,-}={-B-2W-\sqrt{-3\alpha-2y+{2\beta\over W}}\over4}={\scriptstyle-B-2W-\frac23t^4(1+\dots+\bigoo{17})\over4}\cr
&={3\over8}t^5+\dots+\bigoo{21}
\end{align*}
\rbreak
We can check that these expansions for the eigenvalues are correct by checking that evaluating the elementary symmetric polynomials in four variables on them gives the coefficients of the characteristic polynomial.  i.e.
\begin{align*}
&\hbox{characteristic polynomial}=\cr
&x^4-\underbrace{\elem1}_{\tr}x^3+\elem2x^2\cr
&\hskip11em-\elem3x+\underbrace{\elem4}_{\det}
\end{align*}
Maple confirms that
\begin{align*}
\elem1+B&=\bigoo{21}\cr
\elem2-C&=\bigoo{21}\cr
\elem3+D&=\bigoo{24}\cr
\elem4-E&=\bigoo{28}\cr
\end{align*}
\rbreak
To find the associated eigenvectors, we need to find the kernel of $A:=Y-\lll jI$.  i.e. we want to find $v_j$ such that $Av_j=0$.  If the last row of $A$ isn't a row of zeroes, which is the case for the four $A$'s that we examine here, then a sequence of elementary row operations represented by multiplication by an invertible $J$ can take $A$ to a matrix with a row of 0's on the bottom.
$$
JA=\bracketmatrix{&&&a\cr&B&&b\cr&&&c\cr0&0&0&0}
$$
then
\begin{align*}
&JA\bracketmatrix{-B^{-1}\bracketmatrix{a\cr b\cr c}\cr1}=\bracketmatrix{&&&a\cr&B&&b\cr&&&c\cr0&0&0&0}\bracketmatrix{-B^{-1}\bracketmatrix{a\cr b\cr c}\cr1}\cr
&=\bracketmatrix{-BB^{-1}\bracketmatrix{a\cr b\cr c}+\bracketmatrix{a\cr b\cr c}\cr0}=\bracketmatrix{0\cr0\cr0\cr0} \cr
\implies&A\bracketmatrix{-B^{-1}\bracketmatrix{a\cr b\cr c}\cr1}=J^{-1}\bracketmatrix{0\cr0\cr0\cr0}=\bracketmatrix{0\cr0\cr0\cr0}\cr
\end{align*}
so $\bracketmatrix{-B^{-1}\bracketmatrix{a\cr b\cr c}\cr1}$ is an unnormalized eigenvector.  We need $-B^{-1}\bracketmatrix{a\cr b\cr c}$ for each $A$.  To expand $B^{-1}$ we just expand $\hbox{adj}\,B$ and $\det B$ and divide to get $B^{-1}={1\over\det B}\hbox{adj}\,B$.

\rbreak
\subsection{$v_1$:}
\begin{align*}
&A=\cr
&\bracketmatrix{
-8640+ \dots +O \left( {t}^{25} \right) &\scriptstyle 360\,t+ \dots +O \left( {t}^{25} \right) &\scriptstyle -30\,{t}^{2}+ \dots +O \left( {t}^{25} \right) &\scriptstyle-360\,t+ \dots +O \left( {t}^{25} \right) \cr \noalign{\medskip}360\,t+ \dots +O \left( {t}^{25} \right) &\scriptstyle -4320+ \dots +O \left( {t}^{25} \right) &\scriptstyle -360\,t+ \dots +O \left( {t}^{25} \right) &\scriptstyle-4320+ \dots +O \left( {t}^{25} \right) \cr \noalign{\medskip}-30\,{t}^{2}+ \dots +O \left( {t}^{25} \right) &\scriptstyle -360\,t+ \dots +O \left( {t}^{25} \right) &\scriptstyle -8640+ \dots +O \left( {t}^{25} \right) &\scriptstyle360\,t+ \dots +O \left( {t}^{25} \right) \cr \noalign{\medskip}-360\,t+ \dots +O \left( {t}^{25} \right) &\scriptstyle -4320+ \dots +O \left( {t}^{25} \right) &\scriptstyle 360\,t+ \dots +O \left( {t}^{25} \right) &\scriptstyle-4320+ \dots +O \left( {t}^{25} \right)
}
\end{align*}
We get $v_1=$\hfill
$$
\bracketmatrix{
-{\sqrt2\over24}t-{\sqrt2\over48}{t}^{2}+\cdots +\bigoo{25} \cr -{\sqrt2\over38}-{\sqrt2\over8}t+\cdots +\bigoo{25} \cr {\sqrt2\over24}t+0\,t^2+\cdots +\bigoo{25} \cr {\sqrt2\over2}-{\sqrt2\over8}t+\cdots +\bigoo{25}
}
$$
Maple confirms that 
$$
Yv_1-\lll1v_1=\bracketmatrix{\bigoo{25}\cr \bigoo{25}\cr \bigoo{25}\cr \bigoo{25}\cr }
$$
\rbreak
\subsection{$v_2$:}
$$
A=\bracketmatrix{
\scriptstyle30\,{t}^{2}+ \dots +O \left( {t}^{25} \right) &\scriptstyle 360\,t+ \dots +O \left( {t}^{25} \right) &\scriptstyle -30\,{t}^{2}+ \dots +O \left( {t}^{25} \right) &\scriptstyle-360\,t+ \dots +O \left( {t}^{25} \right) \cr \noalign{\medskip}\scriptstyle360\,t+ \dots +O \left( {t}^{25} \right) &\scriptstyle 4320+ \dots +O \left( {t}^{25} \right) &\scriptstyle -360\,t+ \dots +O \left( {t}^{25} \right) &\scriptstyle-4320+ \dots +O \left( {t}^{25} \right) \cr \noalign{\medskip}\scriptstyle-30\,{t}^{2}+ \dots +O \left( {t}^{25} \right) &\scriptstyle -360\,t+ \dots +O \left( {t}^{25} \right) &\scriptstyle 30\,{t}^{2}+ \dots +O \left( {t}^{25} \right) &\scriptstyle360\,t+ \dots +O \left( {t}^{25} \right) \cr \noalign{\medskip}\scriptstyle-360\,t+ \dots +O \left( {t}^{25} \right) &\scriptstyle -4320+ \dots +O \left( {t}^{25} \right) &\scriptstyle 360\,t+ \dots +O \left( {t}^{25} \right) &\scriptstyle4320+ \dots +O \left( {t}^{25} \right)
}
$$
We get $v_2=$\hfill
$$
\bracketmatrix{
-\frac12-{3\over16}\,t+\cdots +\bigoo{18}\cr \frac12-{3\over16}t+\cdots +\bigoo{18}\cr -\frac12+{1\over16}t+\cdots +\bigoo{18}\cr \frac12+{1\over16}t+\cdots +\bigoo{18}
}
$$
again, maple confirms that
$$
Yv_2-\lll2v_2=\bracketmatrix{\bigoo{18}\cr \bigoo{18}\cr \bigoo{18}\cr \bigoo{18}\cr }
$$
\rbreak
\subsection{$v_3$:}
$$
A=\bracketmatrix{
\scriptstyle30\,{t}^{2}+ \dots +O \left( {t}^{21} \right) &\scriptstyle360\,t+ \dots +O \left( {t}^{21} \right) &\scriptstyle-30\,{t}^{2}+ \dots +O \left( {t}^{21} \right) &\scriptstyle-360\,t+ \dots +O \left( {t}^{21} \right) \cr \noalign{\medskip}\scriptstyle360\,t+ \dots +O \left( {t}^{21} \right) &\scriptstyle4320+ \dots +O \left( {t}^{21} \right) &\scriptstyle-360\,t+ \dots +O \left( {t}^{21} \right) &\scriptstyle-4320+ \dots +O \left( {t}^{21} \right) \cr \noalign{\medskip}\scriptstyle-30\,{t}^{2}+ \dots +O \left( {t}^{21} \right) &\scriptstyle-360\,t+ \dots +O \left( {t}^{21} \right) &\scriptstyle30\,{t}^{2}+ \dots +O \left( {t}^{21} \right) &\scriptstyle360\,t+ \dots +O \left( {t}^{21} \right) \cr \noalign{\medskip}\scriptstyle-360\,t+ \dots +O \left( {t}^{21} \right) &\scriptstyle-4320+ \dots +O \left( {t}^{21} \right) &\scriptstyle360\,t+ \dots +O \left( {t}^{21} \right) &\scriptstyle4320+ \dots +O \left( {t}^{21} \right)
}
$$
In this case, the unnormalized eigenvector would have negative powers of $t$, so in the maple calculations, we find the normalized version of $tv$ which is the same.  i.e. we use the fact that $${t\tilde v\over\|t\tilde v\|}={\tilde v\over \|\tilde v\|}$$  We get $v_3=$\hfill
$$
\bracketmatrix{
{\sqrt2\over2}+0\,t+\cdots +\bigoo9\cr {\sqrt2\over12}t+{\frac {61\sqrt2}{576}}{t}^{2}+\cdots +\bigoo9 \cr -{\sqrt2\over2}+0\,t+\cdots +\bigoo9 \cr {\sqrt2\over6}t+{\frac {109\sqrt2}{576}}{t}^{2}+\cdots +\bigoo9
}
$$
again maple confirms that
$$
Yv_3-\lll3v_3=\bracketmatrix{\bigoo9\cr \bigoo9\cr \bigoo9\cr \bigoo9\cr }
$$
\rbreak
\subsection{$v_4$:}
$$
A=\bracketmatrix{
\scriptstyle30\,{t}^{2}+ \dots +O \left( {t}^{21} \right) &\scriptstyle360\,t+ \dots +O \left( {t}^{21} \right) &\scriptstyle-30\,{t}^{2}+ \dots +O \left( {t}^{21} \right) &\scriptstyle-360\,t+ \dots +O \left( {t}^{21} \right) \cr \noalign{\medskip}\scriptstyle360\,t+ \dots +O \left( {t}^{21} \right) &\scriptstyle4320+ \dots +O \left( {t}^{21} \right) &\scriptstyle-360\,t+ \dots +O \left( {t}^{21} \right) &\scriptstyle-4320+ \dots +O \left( {t}^{21} \right) \cr \noalign{\medskip}\scriptstyle-30\,{t}^{2}+ \dots +O \left( {t}^{21} \right) &\scriptstyle-360\,t+ \dots +O \left( {t}^{21} \right) &\scriptstyle30\,{t}^{2}+ \dots +O \left( {t}^{21} \right) &\scriptstyle360\,t+ \dots +O \left( {t}^{21} \right) \cr \noalign{\medskip}\scriptstyle-360\,t+ \dots +O \left( {t}^{21} \right) &\scriptstyle-4320+ \dots +O \left( {t}^{21} \right) &\scriptstyle360\,t+ \dots +O \left( {t}^{21} \right) &\scriptstyle4320+ \dots +O \left( {t}^{21} \right)
}
$$
We get $v_4=$\hfill
$$
\bracketmatrix{
\frac12-{3\over16}t+\cdots+\bigoo{12}\cr \frac12-{1\over16}t+\cdots+\bigoo{12}\cr \frac12+{1\over16}t+\cdots+\bigoo{12}\cr \frac12+{3\over16}t+\cdots+\bigoo{12}
}
$$
maple confirms that
$$
Yv_4-\lll4v_4=\bracketmatrix{\bigoo{12}\cr \bigoo{12}\cr \bigoo{12}\cr \bigoo{12}\cr }
$$
\rbreak
All collected together, we have
\begin{align*}
  \lll1&=8640+\dots+\bigoo{25}\cr
  \lll2&=6t^3+\dots+\bigoo{25}\cr
  \lll3&=\frac13t^4+\dots+\bigoo{21}\cr
  \lll4&=\frac38t^5+\dots+\bigoo{21}\cr
\end{align*}
and $U=[U_{ij}]$ where
\begin{align*}
U_{11}&=-{\sqrt2\over24}t-{\sqrt2\over48}{t}^{2}+\cdots+\bigoo{25} &U_{12}&= -{\sqrt2\over2}-{\sqrt2\over8}t+\cdots+\bigoo{25}\cr
U_{21}&=-{\frac {1}{2}}-{\frac {3}{16}}t+\cdots+\bigoo{18} &U_{22}&= {\frac {1}{2}}-{\frac {3}{16}}t+\cdots+\bigoo{18} \cr
U_{31}&={\sqrt2\over2}+0\,t+\cdots+\bigoo9 &U_{32}&= {\sqrt2\over12}t+{61\sqrt2\over576}t^2+\cdots+\bigoo9\cr
U_{41}&={\frac {1}{2}}-{\frac {3}{16}}t+\cdots+\bigoo{12} &U_{42}&= {\frac {1}{2}}-{\frac {1}{16}}t+\cdots+\bigoo{12}
\end{align*}
\begin{align*}
U_{13}&= {\sqrt2\over24}t+0\,t^2+\cdots+\bigoo{25}&U_{14}&= {\sqrt2\over2}-{\sqrt2\over8}t+\cdots+\bigoo{25}\cr
U_{23}&= -{\frac {1}{2}}+{\frac {1}{16}}t+\cdots+\bigoo{18} &U_{24}&= {\frac {1}{2}}+{\frac {1}{16}}t+\cdots+\bigoo{18}\cr
U_{33}&= -{\sqrt2\over2}+0\,t+\cdots+\bigoo9&U_{34}&= {\sqrt2\over6}t+{109\sqrt2\over576}t^2+\cdots+\bigoo9\cr
U_{43}&= {\frac {1}{2}}+{\frac {1}{16}}t+\cdots+\bigoo{12}&U_{44}&= {\frac {1}{2}}+{\frac {3}{16}}t+\cdots+\bigoo{12}
\end{align*}
The calculation only requires the $\lambda_i$ be expanded to $t^{12}$ and $U$ be expanded to $t^3$.

\begin{acknowledgements}
I would like to thank my advisor Bernard Shiffman for his advice, patience, encouragement, and many ideas about solving this problem.
\end{acknowledgements}

\bibliographystyle{spmpsci}      
\bibliography{paper}   

\begin{thebibliography}{10}
\providecommand{\url}[1]{{#1}}
\providecommand{\urlprefix}{URL }
\expandafter\ifx\csname urlstyle\endcsname\relax
  \providecommand{\doi}[1]{DOI~\discretionary{}{}{}#1}\else
  \providecommand{\doi}{DOI~\discretionary{}{}{}\begingroup
  \urlstyle{rm}\Url}\fi

\bibitem{arxiv_version}
Baber, J.: Scaled correlations of critical points of random sections on riemann
  surfaces (2012).
\newblock \urlprefix\url{http://arxiv.org/abs/1106.4737}

\bibitem{pl}
Bleher, P., Shiffman, B., Zelditch, S.: Poincar\'e-{L}elong approach to
  universality and scaling of correlations between zeros.
\newblock Comm. Math. Phys. \textbf{208}(3), 771--785 (2000).
\newblock \doi{10.1007/s002200050010}.
\newblock \urlprefix\url{http://dx.doi.org/10.1007/s002200050010}

\bibitem{bsz}
Bleher, P., Shiffman, B., Zelditch, S.: Universality and scaling of
  correlations between zeros on complex manifolds.
\newblock Invent. Math. \textbf{142}(2), 351--395 (2000).
\newblock \doi{10.1007/s002220000092}.
\newblock \urlprefix\url{http://dx.doi.org/10.1007/s002220000092}

\bibitem{symplectic}
Bleher, P., Shiffman, B., Zelditch, S.: Universality and scaling of zeros on
  symplectic manifolds.
\newblock In: Random matrix models and their applications, \emph{Math. Sci.
  Res. Inst. Publ.}, vol.~40, pp. 31--69. Cambridge Univ. Press, Cambridge
  (2001)

\bibitem{bloom}
Bloom, T.: Random polynomials and {G}reen functions.
\newblock Int. Math. Res. Not. (28), 1689--1708 (2005).
\newblock \doi{10.1155/IMRN.2005.1689}.
\newblock \urlprefix\url{http://dx.doi.org/10.1155/IMRN.2005.1689}

\bibitem{bs}
Bloom, T., Shiffman, B.: Zeros of random polynomials on {$\Bbb C^m$}.
\newblock Math. Res. Lett. \textbf{14}(3), 469--479 (2007)

\bibitem{cardano}
Cardano, G.: Ars magna or {T}he rules of algebra.
\newblock Dover Publications Inc., New York (1993).
\newblock Translated from the Latin and edited by T. Richard Witmer, With a
  foreword by Oystein Ore, Reprint of the 1968 edition

\bibitem{cpsv1}
Douglas, M.R., Shiffman, B., Zelditch, S.: Critical points and supersymmetric
  vacua. {I}.
\newblock Comm. Math. Phys. \textbf{252}(1-3), 325--358 (2004).
\newblock \doi{10.1007/s00220-004-1228-y}.
\newblock \urlprefix\url{http://dx.doi.org/10.1007/s00220-004-1228-y}

\bibitem{cpsv2}
Douglas, M.R., Shiffman, B., Zelditch, S.: Critical points and supersymmetric
  vacua. {II}. {A}symptotics and extremal metrics.
\newblock J. Differential Geom. \textbf{72}(3), 381--427 (2006).
\newblock
  \urlprefix\url{http://projecteuclid.org/getRecord?id=euclid.jdg/1143593745}

\bibitem{cpsv3}
Douglas, M.R., Shiffman, B., Zelditch, S.: Critical points and supersymmetric
  vacua. {III}. {S}tring/{M} models.
\newblock Comm. Math. Phys. \textbf{265}(3), 617--671 (2006).
\newblock \doi{10.1007/s00220-006-0003-7}.
\newblock \urlprefix\url{http://dx.doi.org/10.1007/s00220-006-0003-7}

\bibitem{gh}
Griffiths, P., Harris, J.: Principles of algebraic geometry.
\newblock Wiley Classics Library. John Wiley \& Sons Inc., New York (1994).
\newblock Reprint of the 1978 original

\bibitem{hammersley}
Hammersley, J.M.: The zeros of a random polynomial.
\newblock In: Proceedings of the {T}hird {B}erkeley {S}ymposium on
  {M}athematical {S}tatistics and {P}robability, 1954--1955, vol. {II}, pp.
  89--111. University of California Press, Berkeley and Los Angeles (1956)

\bibitem{hannay}
Hannay, J.H.: Chaotic analytic zero points: exact statistics for those of a
  random spin state.
\newblock J. Phys. A \textbf{29}(5), L101--L105 (1996).
\newblock \doi{10.1088/0305-4470/29/5/004}.
\newblock \urlprefix\url{http://dx.doi.org/10.1088/0305-4470/29/5/004}

\bibitem{kac}
Kac, M.: On the average number of real roots of a random algebraic equation.
  {II}.
\newblock Proc. London Math. Soc. (2) \textbf{50}, 390--408 (1949)

\bibitem{riceold}
Rice, S.O.: The {D}istribution of the {M}axima of a {R}andom {C}urve.
\newblock Amer. J. Math. \textbf{61}(2), 409--416 (1939).
\newblock \doi{10.2307/2371510}.
\newblock \urlprefix\url{http://dx.doi.org/10.2307/2371510}

\bibitem{rice}
Rice, S.O.: Mathematical analysis of random noise.
\newblock Bell System Tech. J. \textbf{23}, 282--332 (1944)

\bibitem{numbervariance}
Shiffman, B., Zelditch, S.: Number variance of random zeros on complex
  manifolds.
\newblock Geom. Funct. Anal. \textbf{18}(4), 1422--1475 (2008).
\newblock \doi{10.1007/s00039-008-0686-3}.
\newblock \urlprefix\url{http://dx.doi.org/10.1007/s00039-008-0686-3}

\bibitem{sodin4}
Sodin, M.: Zeroes of {G}aussian analytic functions.
\newblock In: European {C}ongress of {M}athematics, pp. 445--458. Eur. Math.
  Soc., Z\"urich (2005)

\bibitem{sodin1}
Sodin, M., Tsirelson, B.: Random complex zeroes. {I}. {A}symptotic normality.
\newblock Israel J. Math. \textbf{144}, 125--149 (2004).
\newblock \doi{10.1007/BF02984409}.
\newblock \urlprefix\url{http://dx.doi.org/10.1007/BF02984409}

\bibitem{sodin3}
Sodin, M., Tsirelson, B.: Random complex zeroes. {III}. {D}ecay of the hole
  probability.
\newblock Israel J. Math. \textbf{147}, 371--379 (2005).
\newblock \doi{10.1007/BF02785373}.
\newblock \urlprefix\url{http://dx.doi.org/10.1007/BF02785373}

\bibitem{sodin2}
Sodin, M., Tsirelson, B.: Random complex zeroes. {II}. {P}erturbed lattice.
\newblock Israel J. Math. \textbf{152}, 105--124 (2006).
\newblock \doi{10.1007/BF02771978}.
\newblock \urlprefix\url{http://dx.doi.org/10.1007/BF02771978}

\end{thebibliography}

\end{document}